%% file: main.tex
\documentclass{amsart}
\usepackage[english]{babel}
\input{macro}

\title{Immersions of complexes of groups}
\author{Jagerynn Ting Verano}
\address{Department of Mathematics, Statistics, and Computer Science,
University of Illinois Chicago}
\email{jveran2@uic.edu}
\urladdr{http://sites.google.com/view/jagerynn/}
\begin{document}

\begin{abstract}
    Given a complex of groups, we construct a new class of complex of groups that records its local data and offer a functorial perspective on the statement that complexes of groups are locally developable. We also construct a new notion of an immersion of complexes of groups and establish that a locally isometric immersion of a complex of groups into a non-positively curved complex of groups is $\pi_1$-injective. Furthermore, the domain complex of groups is developable and the induced map on geometric realizations of developments is an isometric embedding.
\end{abstract}

\maketitle
\setcounter{tocdepth}{1}
\tableofcontents
\section{Introduction}
In \cite{Agol}, Agol proved that every hyperbolic group admitting a geometric action on a CAT(0) cube complex is virtually special. As a result, every cubulated hyperbolic group inherits desirable properties such as linearity over $\mathbb{Z}$, subgroup separability and residual finiteness \cite{HaglundFrédéric2008SCC}. In recent years, \textit{complexes of groups} have been utilized to extend these results to more general types of actions, such as hyperbolic groups acting improperly \cite{Groves_2023} and relatively geometric actions \cite{relgeom, EinsteinGrovesNg}. Already we see applications to a broader class of groups, including certain relatively hyperbolic K\"{a}hler groups \cite{BregmanCorey2024Rgao}, hyperbolic hyperbolic-by-cyclic groups \cite{DahmaniFrançois2025Hhga}, a large class of 2-dimensional Shephard groups\footnote{In short, Shephard groups are specific quotients of Artin groups.} \cite{goldman20242dimensionalshephardgroups}, and closed aspherical manifolds obtained from relative strict hyperbolization \cite{LafontRuffoni25}.

The theory of complexes of groups generalizes covering theory. Like topological spaces, a \textit{developable} complex of groups arises from a group action on a simply connected polyhedral complex. Its data allows one to construct the complex on which its \textit{fundamental group} acts. In contrast to covering theory, the group action need not be free. Hence, in addition to topological data coming from the quotient space, the fundamental group of a developable complex of groups contains cell stabilizers.

A complex of groups is not always developable. Non-developability is a global phenomenon. Locally, however, it is always developable (see \cite[III.$\mathcal{C}$.4, pp. 520, 555]{bh}). Our first goal is to interpret this statement from a functorial perspective (formal definitions are introduced in Sections \ref{sect_CoGbasics} and \ref{sect_localCoG}):

\begin{restatable}{theorem}{thmlocalCoG}\label{thm_localCoG}
    Let $G(\mathcal{Y})=(G_\sigma, \psi_a, g_{a,b})$ be a complex of groups over a scwol $\mathcal{Y}$, and let $\gamma\in V(\mathcal{Y})$. The local complex of groups $L(\mathcal{Y}(\gamma))$ over $\gamma$ has the following properties:
    \begin{enumerate}
        \item It is developable,
        \item The fundamental group of $L(\mathcal{Y}(\gamma))$ is the local group $G_\gamma$, and
        \item The development $D(\mathcal{Y}(\gamma),\iota_T)$ is the local development $\mathcal{Y}(\tilde{\gamma})$.
    \end{enumerate}
\end{restatable}

In Construction \ref{construction_morphismfunctorial}, we describe a morphism $\Sigma:L(\mathcal{Y}(\gamma))\to G(\mathcal{Y})$ for each $\gamma\in V(\mathcal{Y})$.
We utilize this construction to obtain a functorial version of the developability criterion (see \cite[Corollary III.$\mathcal{C}$.2.15]{bh} for this criterion).
\begin{restatable}{corollary}{corthmlocalCoG}\label{cor:thm_localCoG}
Let $G(\mathcal{Y})$ be a complex of groups over a scwol $\mathcal{Y}$. Then $G(\mathcal{Y})$ is developable if and only if for all $\gamma\in V(\mathcal{Y})$, $$\Sigma:L(\mathcal{Y}(\gamma))\longrightarrow G(\mathcal{Y})$$ as defined in Construction \ref{construction_morphismfunctorial} is $\pi_1-$injective.
\end{restatable}

Our second goal is to develop a new notion of an \textit{immersion of complexes of groups}. In \cite[Definition 3.1]{amar}, Martin defines an immersion of complexes of groups as a morphism that is injective on local groups and whose underlying map is a simplicial immersion. By Martin's definition, most coverings of complexes of groups are not immersions.
In Definition \ref{def_immersion}, we introduce a more general notion of an immersion. 

We apply our notion of a locally isometric immersion to prove a corollary of the Cartan--Hadamard theorem for complexes of groups \cite[Theorem III.$\mathcal{C}$.4.17]{bh}. Specifically, \cite[Theorem III.$\mathcal{C}$.4.17]{bh} states that every \textit{non-positively curved complex of groups} is developable (see \cite[Theorem III.$\mathcal{C}$.4.17]{bh}).  Our theorem takes inspiration from \cite[Proposition II.4.14]{bh} and generalizes this local-to-global proposition from metric spaces to complexes of groups.
\begin{restatable}{theorem}{thmimmersion}
\label{thm_immersion}
Let $H(\mathcal{Y})$ and $G(\mathcal{X})$ be complexes of groups over connected scwols $\mathcal{Y}$ and $\mathcal{X}$ respectively.
If $G(\mathcal{X})$ is non-positively curved and $\phi:H(\mathcal{Y}) \to G(\mathcal{X})$ is a locally isometric immersion over a morphism of scwols $f: \mathcal{Y} \to \mathcal{X}$, then $H(\mathcal{Y})$ is also non-positively curved and hence developable. Moreover,
\begin{enumerate}
    \item \label{pi1} The induced map on fundamental groups is $\pi_1-$injective, and
    \item The elevation of geometric realizations is an isometric embedding.
\end{enumerate}  
\end{restatable} 

Locally isometric immersions arise naturally from group actions on CAT(0) cube complexes.
Let $X$ be a CAT(0) cube complex and let $W\subset X$ be a hyperplane. The action of a group $G$ on $X$ induces an immersion of complexes of groups over the map of quotient complexes $$\leftQ{W}{\text{Stab}(W)}\longrightarrow \leftQ{X}{G}.$$ Since $X$ is CAT(0), the complex of groups associated to the action of $G$ on $X$ is non-positively curved. By the convexity of $W$ in $X$, the hyperplane complex of groups inherits the non-positively curved metric of $X$ and is developable. Finally, the induced map on fundamental groups is the natural inclusion $\text{Stab}(W)\hookrightarrow G$.
Theorem \ref{thm_immersion} allows us to make this conclusion from an arbitrary locally isometric immersion into a non-positively curved complex of groups.

\section{Complexes of groups} 
\label{sect_CoGbasics}
In this section we collect the necessary terminology and facts about complexes of groups that we need for the rest of the paper. We begin by introducing the structure underlying a complex of groups.

\subsection{Small categories without loops}
The following definitions can be found in \cite[III.$\mathcal{C}$.1]{bh} in greater detail. 
\begin{definition}
\label{scwol}
    A \textit{small category without loops} (or a \textit{scwol}), $\mathcal{Y}$, is a category with objects $V(\mathcal{Y})$ and non-identity morphisms $E(\mathcal{Y})$, such that every object $\sigma\in V(\mathcal{Y})$,
    $$\text{Mor}(\sigma,\sigma)=\{\text{id}_\sigma\}.$$
    If $a\in \text{Mor}(\sigma,\tau)$, we say that
    $i(a)=\sigma$ and $t(a)=\tau.$ We denote by $E^{(k)}(\mathcal{Y})$ the collection of tuples of morphisms $(a_1,...,a_k)\in E(\mathcal{Y})\times ...\times E(\mathcal{Y})$, where $t(a_{j})=i(a_{j+1})$ for $j\in \{1,2,...,k-1\}$.
\end{definition}

\begin{definition}
    Let $\mathcal{Y}$ be a scwol and let each $(a_1,a_2,...,a_k)\in E^{(k)}(\mathcal{Y})$ correspond to a $(k+1)-$simplex with edges indexed $|a_i|$ and $|a_ia_{i+1}|$ by morphisms and their compositions, such that $$i(|a_ia_{i+1}|)=i(|a_{i+1}|)=|i(a_{i+1})| \text{ and }t(|a_ia_{i+1}|)=t(|a_i|)=|t(a_i)|.$$
    The \textit{geometric realization of a scwol $\mathcal{Y}$}, denoted by $|\mathcal{Y}|$, is the quotient of all such simplices by relations in $\mathcal{Y}$. In particular, it is a polyhedral complex with vertices indexed by $V(\mathcal{Y})$ and morphisms indexed by $E(\mathcal{Y})$.
\end{definition}

Just as a scwol gives rise to a polyhedral complex via its geometric realization, a polyhedral complex gives rise to a scwol in a natural way: its objects correspond faces of the polyhedral complex and its morphisms are defined by reverse inclusions of faces. More precisely, for every relation $\tau\supset \sigma$ between two faces, there exists a morphism $\tau\to \sigma$. The next example illustrates this concept.

\begin{example}
The scwol naturally associated to a 2-simplex is given by:
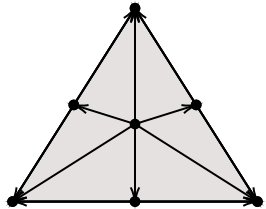
\begin{center}
\tikzset{every picture/.style={line width=0.75pt}} 

\begin{tikzpicture}[x=0.7pt,y=0.7pt,yscale=-.7,xscale=.7]

\draw  [fill={rgb, 255:red, 230; green, 225; blue, 225 }  ,fill opacity=1 ] (288.5,48.5) -- (383,197.83) -- (194,197.83) -- cycle ;
\draw [fill={rgb, 255:red, 0; green, 0; blue, 0 }  ,fill opacity=1 ]   (288.5,137.95) -- (288.5,195.83) ;
\draw [shift={(288.5,197.83)}, rotate = 270] [color={rgb, 255:red, 0; green, 0; blue, 0 }  ][line width=0.75]    (10.93,-3.29) .. controls (6.95,-1.4) and (3.31,-0.3) .. (0,0) .. controls (3.31,0.3) and (6.95,1.4) .. (10.93,3.29)   ;
\draw [shift={(288.5,137.95)}, rotate = 90] [color={rgb, 255:red, 0; green, 0; blue, 0 }  ][fill={rgb, 255:red, 0; green, 0; blue, 0 }  ][line width=0.75]      (0, 0) circle [x radius= 3.35, y radius= 3.35]   ;
\draw    (288.5,137.95) -- (243.16,123.76) ;
\draw [shift={(241.25,123.17)}, rotate = 17.37] [color={rgb, 255:red, 0; green, 0; blue, 0 }  ][line width=0.75]    (10.93,-3.29) .. controls (6.95,-1.4) and (3.31,-0.3) .. (0,0) .. controls (3.31,0.3) and (6.95,1.4) .. (10.93,3.29)   ;
\draw    (288.5,137.95) -- (333.84,123.76) ;
\draw [shift={(335.75,123.17)}, rotate = 162.63] [color={rgb, 255:red, 0; green, 0; blue, 0 }  ][line width=0.75]    (10.93,-3.29) .. controls (6.95,-1.4) and (3.31,-0.3) .. (0,0) .. controls (3.31,0.3) and (6.95,1.4) .. (10.93,3.29)   ;
\draw    (288.5,137.95) -- (288.5,50.5) ;
\draw [shift={(288.5,48.5)}, rotate = 90] [color={rgb, 255:red, 0; green, 0; blue, 0 }  ][line width=0.75]    (10.93,-3.29) .. controls (6.95,-1.4) and (3.31,-0.3) .. (0,0) .. controls (3.31,0.3) and (6.95,1.4) .. (10.93,3.29)   ;
\draw    (288.5,137.95) -- (381.31,196.76) ;
\draw [shift={(383,197.83)}, rotate = 212.36] [color={rgb, 255:red, 0; green, 0; blue, 0 }  ][line width=0.75]    (10.93,-3.29) .. controls (6.95,-1.4) and (3.31,-0.3) .. (0,0) .. controls (3.31,0.3) and (6.95,1.4) .. (10.93,3.29)   ;
\draw    (288.5,48.5) -- (195.07,196.14) ;
\draw [shift={(194,197.83)}, rotate = 302.33] [color={rgb, 255:red, 0; green, 0; blue, 0 }  ][line width=0.75]    (10.93,-3.29) .. controls (6.95,-1.4) and (3.31,-0.3) .. (0,0) .. controls (3.31,0.3) and (6.95,1.4) .. (10.93,3.29)   ;
\draw [shift={(241.25,123.17)}, rotate = 122.33] [color={rgb, 255:red, 0; green, 0; blue, 0 }  ][fill={rgb, 255:red, 0; green, 0; blue, 0 }  ][line width=0.75]      (0, 0) circle [x radius= 3.35, y radius= 3.35]   ;
\draw [shift={(288.5,48.5)}, rotate = 122.33] [color={rgb, 255:red, 0; green, 0; blue, 0 }  ][fill={rgb, 255:red, 0; green, 0; blue, 0 }  ][line width=0.75]      (0, 0) circle [x radius= 3.35, y radius= 3.35]   ;
\draw    (288.5,48.5) -- (381.93,196.14) ;
\draw [shift={(383,197.83)}, rotate = 237.67] [color={rgb, 255:red, 0; green, 0; blue, 0 }  ][line width=0.75]    (10.93,-3.29) .. controls (6.95,-1.4) and (3.31,-0.3) .. (0,0) .. controls (3.31,0.3) and (6.95,1.4) .. (10.93,3.29)   ;
\draw [shift={(335.75,123.17)}, rotate = 57.67] [color={rgb, 255:red, 0; green, 0; blue, 0 }  ][fill={rgb, 255:red, 0; green, 0; blue, 0 }  ][line width=0.75]      (0, 0) circle [x radius= 3.35, y radius= 3.35]   ;
\draw [shift={(288.5,48.5)}, rotate = 57.67] [color={rgb, 255:red, 0; green, 0; blue, 0 }  ][fill={rgb, 255:red, 0; green, 0; blue, 0 }  ][line width=0.75]      (0, 0) circle [x radius= 3.35, y radius= 3.35]   ;
\draw    (383,197.83) -- (196,197.83) ;
\draw [shift={(194,197.83)}, rotate = 360] [color={rgb, 255:red, 0; green, 0; blue, 0 }  ][line width=0.75]    (10.93,-3.29) .. controls (6.95,-1.4) and (3.31,-0.3) .. (0,0) .. controls (3.31,0.3) and (6.95,1.4) .. (10.93,3.29)   ;
\draw [shift={(288.5,197.83)}, rotate = 180] [color={rgb, 255:red, 0; green, 0; blue, 0 }  ][fill={rgb, 255:red, 0; green, 0; blue, 0 }  ][line width=0.75]      (0, 0) circle [x radius= 3.35, y radius= 3.35]   ;
\draw [shift={(383,197.83)}, rotate = 180] [color={rgb, 255:red, 0; green, 0; blue, 0 }  ][fill={rgb, 255:red, 0; green, 0; blue, 0 }  ][line width=0.75]      (0, 0) circle [x radius= 3.35, y radius= 3.35]   ;
\draw    (194,197.83) -- (381,197.83) ;
\draw [shift={(383,197.83)}, rotate = 180] [color={rgb, 255:red, 0; green, 0; blue, 0 }  ][line width=0.75]    (10.93,-3.29) .. controls (6.95,-1.4) and (3.31,-0.3) .. (0,0) .. controls (3.31,0.3) and (6.95,1.4) .. (10.93,3.29)   ;
\draw [shift={(288.5,197.83)}, rotate = 0] [color={rgb, 255:red, 0; green, 0; blue, 0 }  ][fill={rgb, 255:red, 0; green, 0; blue, 0 }  ][line width=0.75]      (0, 0) circle [x radius= 3.35, y radius= 3.35]   ;
\draw [shift={(194,197.83)}, rotate = 0] [color={rgb, 255:red, 0; green, 0; blue, 0 }  ][fill={rgb, 255:red, 0; green, 0; blue, 0 }  ][line width=0.75]      (0, 0) circle [x radius= 3.35, y radius= 3.35]   ;
\draw    (383,197.83) -- (289.57,50.19) ;
\draw [shift={(288.5,48.5)}, rotate = 57.67] [color={rgb, 255:red, 0; green, 0; blue, 0 }  ][line width=0.75]    (10.93,-3.29) .. controls (6.95,-1.4) and (3.31,-0.3) .. (0,0) .. controls (3.31,0.3) and (6.95,1.4) .. (10.93,3.29)   ;
\draw [shift={(335.75,123.17)}, rotate = 237.67] [color={rgb, 255:red, 0; green, 0; blue, 0 }  ][fill={rgb, 255:red, 0; green, 0; blue, 0 }  ][line width=0.75]      (0, 0) circle [x radius= 3.35, y radius= 3.35]   ;
\draw [shift={(383,197.83)}, rotate = 237.67] [color={rgb, 255:red, 0; green, 0; blue, 0 }  ][fill={rgb, 255:red, 0; green, 0; blue, 0 }  ][line width=0.75]      (0, 0) circle [x radius= 3.35, y radius= 3.35]   ;
\draw [fill={rgb, 255:red, 155; green, 155; blue, 155 }  ,fill opacity=1 ]   (194,197.83) -- (287.43,50.19) ;
\draw [shift={(288.5,48.5)}, rotate = 122.33] [color={rgb, 255:red, 0; green, 0; blue, 0 }  ][line width=0.75]    (10.93,-3.29) .. controls (6.95,-1.4) and (3.31,-0.3) .. (0,0) .. controls (3.31,0.3) and (6.95,1.4) .. (10.93,3.29)   ;
\draw [shift={(241.25,123.17)}, rotate = 302.33] [color={rgb, 255:red, 0; green, 0; blue, 0 }  ][fill={rgb, 255:red, 0; green, 0; blue, 0 }  ][line width=0.75]      (0, 0) circle [x radius= 3.35, y radius= 3.35]   ;
\draw [shift={(194,197.83)}, rotate = 302.33] [color={rgb, 255:red, 0; green, 0; blue, 0 }  ][fill={rgb, 255:red, 0; green, 0; blue, 0 }  ][line width=0.75]      (0, 0) circle [x radius= 3.35, y radius= 3.35]   ;
\draw    (288.5,137.95) -- (195.69,196.76) ;
\draw [shift={(194,197.83)}, rotate = 327.64] [color={rgb, 255:red, 0; green, 0; blue, 0 }  ][line width=0.75]    (10.93,-3.29) .. controls (6.95,-1.4) and (3.31,-0.3) .. (0,0) .. controls (3.31,0.3) and (6.95,1.4) .. (10.93,3.29)   ;

\end{tikzpicture}
    \end{center}
\end{example}

\begin{definition}\label{def_morphism_scwols}
Let $f:\mathcal{Y}\to \mathcal{X}$ be a relation between scwols such that for all $\sigma\in V(\mathcal{Y})$, there is a bijection of the set $$\{a\in E(\mathcal{Y})|i(a)=\sigma\}$$ onto the set $$\{\bar{a}\in E(\mathcal{X})|i(\bar{a})=f(\sigma)\}.$$ We refer to $f$ as a \textit{(non-degenerate) morphism of scwols}.
\end{definition}

\begin{definition}
    The \textit{action of a group $G$ on a scwol $\mathcal{X}$} is a homomorphism $G\to \text{Aut}(\mathcal{X})$ such that:
    \begin{enumerate}
        \item If $g\in \text{Stab}(i(a))$, then $g\in \text{Stab}_G(a)$, and
        \item There are no inversions, i.e. $g.i(a)\neq t(a)$.
    \end{enumerate}
\end{definition}

A morphism $f:\mathcal{Y}\to \mathcal{X}$ of scwols induces a map $|f|:|\mathcal{Y}|\to |\mathcal{X}|$ on geometric realizations. We refer to $|f|$ as the \textit{geometric realization} of $f$. The restriction of $|f|$ to each simplex of $|\mathcal{Y}|$ is affine. Non-degeneracy of $f$ implies that $|f|$ is a homeomorphism on the interiors of simplices.

\begin{definition}\label{def_CoG}
A \textit{complex of groups $G(\mathcal{Y})=(G_\sigma, \psi_a, g_{a,b})$ over a scwol $\mathcal{Y}$} is a set of data comprising
\begin{enumerate}
\item A \textit{local group} $G_\sigma$ for each $\sigma\in V(\mathcal{Y})$,
\item An injective homomorphism $\psi_a:G_{i(a)}\to G_{t(a)}$ for each $a\in E(\mathcal{Y})$, and
\item A \textit{twisting element} $g_{a,b}\in G_{t(a)}$ for each $(a,b)\in E^{(2)}(\mathcal{Y})$, satisfying the following compatibility conditions:\begin{enumerate}
    \item $\text{Ad}(g_{a,b})\psi_{ab}=\psi_a \psi_b$\footnote{We define $\text{Ad}:G\to G$ as $\text{Ad}(g):h\mapsto ghg^{-1}, g\in G$.}, and
    \item $\psi_a(g_{b,c})g_{a,bc}=g_{a,b}g_{ab,c}$ for all $ (a,b,c)\in E^{(3)}(\mathcal{Y})$.
\end{enumerate}
\end{enumerate}
\end{definition}

\begin{definition}\label{morphism_CoG}
    A \textit{morphism $\phi: H(\mathcal{Y})\to G(\mathcal{X})$ of complexes of groups} over a morphism $f:\mathcal{Y}\to \mathcal{X}$ of scwols consists of a \textit{local homomorphism} $\phi_\sigma:H_\sigma\to G_{f(\sigma)}$ for each $\sigma\in V(\mathcal{Y})$ and an element $\phi(a)\in G_{t(f(a))}$ for each $a\in E(\mathcal{Y})$, such that
    \begin{enumerate}
        \item $\text{Ad}(\phi(a))\psi_{f(a)}\phi_{i(a)}=\phi_{t(a)}\psi_a,$ and
        \item $\phi_{t(a)}(g_{a,b})\phi(ab)=\phi(a)\psi_{f(a)}(\phi(b))g_{f(a), f(b)}, $ for all $ (a,b)\in E^{(2)}(\mathcal{Y})$.
    \end{enumerate}
\end{definition}

\begin{construction}\label{coboundary} 
    Let $H(\mathcal{Y})=(H_\sigma, \psi_a, h_{a,b})$ be a complex of groups over $\mathcal{Y}$. For each $a\in E(\mathcal{Y})$, pick an element $g_a$ such that $$\phi:=(g_a): H(\mathcal{Y})\longrightarrow G(\mathcal{Y})$$ 
    is an isomorphism of complexes of groups defined by $\phi_\sigma= \text{id}_{G_\sigma}$ and $\phi(a)= {g_a}^{-1}$. A quick check shows that $$G(\mathcal{Y})=(G_\sigma, \text{Ad}({g_a}^{-1}) \psi_a, \text{Ad}({g_a}^{-1}) \psi_a({g_b}^{-1}) g_{a,b})$$ is deduced from $H(\mathcal{Y})$ by a coboundary\footnote{This terminology is introduced in \cite[III.$\mathcal{C}$, p.535]{bh}.} of $\phi=(g_a)$.
\end{construction}

\begin{definition}\label{def_morphismtoG}
    A \textit{morphism $\phi: H(\mathcal{Y})\to G$ from a complex of groups to a group $G$}
    consists of a \textit{local homomorphism} $\phi_\sigma:H_\sigma\to G$ for each $\sigma\in V(\mathcal{Y})$ and an element $\phi(a)\in G$ for each $a\in E(\mathcal{Y})$ such that
    \begin{enumerate}
        \item $\text{Ad}(\phi(a))\phi_{i(a)}=\phi_{t(a)}\psi_a$, and
        \item $\phi_{t(a)}(h_{a,b})\phi(ab)=\phi(a)\phi(b),$ where $(a,b)\in E^{(2)}(\mathcal{Y})$ and $h_{a,b}$ is the twisting element associated to $(a,b)$.
    \end{enumerate}
\end{definition}

\subsection{The fundamental group of a complex of groups}\label{fundamentalgroup}
The following definition comes from \cite[Theorem III.$\mathcal{C}$.3.7]{bh}.
\begin{definition}\label{def_pi1_presentation}
    Let $G(\mathcal{Y})$ be a complex of groups over a scwol $\mathcal{Y}$ and let $T$ be a choice of maximal tree in $|\mathcal{Y}|^{(1)}$. The \textit{fundamental group} $\pi_1(G(\mathcal{Y}),T)$ is generated by the set
    $$\left\{\bigsqcup_{\sigma\in V(\mathcal{Y})} G_\sigma \bigsqcup E^{\pm}(\mathcal{Y})\right\},$$
    subject to the relations
    
    $$\left\{\begin{aligned} 
    &\text{the set of relations in } G_\sigma &\forall\sigma\in V(\mathcal{Y}),\\
    &(a^{\pm})^{-1}=a^{\mp} &\forall a\in E(\mathcal{Y}),\\
    &a^+b^+= g_{a,b}(ab)^+ &\forall (a,b)\in E^{(2)}(\mathcal{Y}),\\
    &\psi_a(g)=a^+ga^- &\forall g\in G_{i(a)},\\
    &a^{\pm}=1 &\forall |a|\in T.
    \end{aligned}
    \right\}.$$
\end{definition}

\begin{remark}\label{remark_naturalmorphismtopi1}
    By \cite[III.$\mathcal{C}$, p. 553]{bh}, there exists a natural morphism $$\iota_T:G(\mathcal{Y})\longrightarrow \pi_1(G(\mathcal{Y}),T)$$ defined by sending the local group $G_\sigma$ to its natural image given by the above presentation and $a\in E(\mathcal{Y})$ to $a^+\in \pi_1(G(\mathcal{Y}),T)$.
\end{remark}

Like maps between topological spaces, morphisms of complexes of groups induce homomorphisms on the level of fundamental groups. The following proposition describes this using an alternative definition of the fundamental group \cite[Definition III.$\mathcal{C}$.3.5]{bh}. We denote this by $\pi_1(G(\mathcal{Y}),\sigma)$ for a complex of groups $G(\mathcal{Y})$ and $\sigma\in V(\mathcal{Y})$. The equivalence of $\pi_1(G(\mathcal{Y}),\sigma)$ and $\pi_1(G(\mathcal{Y}),T)$ is stated in \cite[Theorem III.$\mathcal{C}$.3.7]{bh}.
\begin{proposition}\cite[Proposition III.$\mathcal{C}$.3.6]{bh}\label{prop_pi1induced}
    A morphism $\phi:H(\mathcal{Y})\to G(\mathcal{X})$ of complexes of groups over a morphism $f:\mathcal{Y}\to \mathcal{X}$ of scwols induces a homomorphism
    $$\phi_*: \pi_1(H(\mathcal{Y}),\sigma)\longrightarrow \pi_1(G(\mathcal{X}),f(\sigma))$$ on the level of fundamental groups, described by
    \begin{align*}
    h&\longmapsto \phi_\sigma(h) &h\in H_\sigma, \sigma\in V(\mathcal{Y})\\
    a^+&\longmapsto \phi(a)f(a)^+ &a\in E(\mathcal{Y})
    \end{align*}
\end{proposition}

The following proposition is an immediate consequence of \cite[Proposition III.$\mathcal{C}$.3.10(1)-(2)]{bh}.
\begin{proposition}\label{prop_pi1inducedtoG}
    Let $\phi: H(\mathcal{Y})\to G$ be a morphism such that $\phi(a)=e$ for all $|a|\in T$. There exists a homomorphism $$\phi_*:\pi_1(H(\mathcal{Y}),T)\longrightarrow G$$
    described by
    \begin{align*}
        h&\longmapsto \phi_\sigma(h) &h\in H_\sigma, \sigma\in V(\mathcal{Y})\\
        a^+&\longmapsto \phi(a) &a\in E(\mathcal{Y})
    \end{align*}
\end{proposition}

\subsection{Developability}\label{section_dev}
\begin{definition}
A complex of groups is \textit{developable} if it arises from a group acting on a simply connected scwol.
\end{definition}

The data of a developable complex of groups describes a group action. In particular, its local groups are precisely the stabilizers of objects up to conjugation, and injective homomorphisms correspond to inclusions of stabilizers. We refer the reader to \cite[Definition III.$\mathcal{C}$.2.9(1)]{bh} for more details on a complex of groups and morphism associated to a group action.

To determine developability, one may apply the developability criterion:
\begin{theorem}
\cite[Corollary III.$\mathcal{C}$.2.15]{bh}
\label{thm_developability_criterion}
    A complex of groups $H(\mathcal{Y})$ over a scwol $\mathcal{Y}$ is developable if and only if
    there exists some group $G$ and some morphism $\phi:H(\mathcal{Y})\to G$ that is injective on local groups, i.e. for all $\sigma\in V(\mathcal{Y})$,
    $\phi_\sigma: H_\sigma\to  G$
    is injective.
\end{theorem}

\begin{definition}
    For every morphism $\phi:H(\mathcal{Y})\to G$ from a complex of groups $H(\mathcal{Y})$ to an arbitrary group $G$, there exists a scwol $D(\mathcal{Y},\phi)$ with a compatible $G-$action. We refer to $D(\mathcal{Y},\phi)$ as the \textit{development of $\mathcal{Y}$ with respect to $\phi$}.
\end{definition} The construction of $D(\mathcal{Y},\phi)$ may be found in \cite[Theorem III.$\mathcal{C}$.2.13]{bh}.

The following theorem is an immediate consequence of \cite[Theorem III.$\mathcal{C}$.2.13]{bh} and \cite[Theorem III.$\mathcal{C}$.3.13]{bh}. It
illustrates the significance of the development in the case where the complex of groups is developable.

\begin{theorem}
\label{thm_D(Y,iota_T)}
    If $G(\mathcal{Y})$ is a developable complex of groups, the development $D(\mathcal{Y}, \iota_T)$ is simply connected. Furthermore, $G(\mathcal{Y})$ and $\iota_T:G(\mathcal{Y})\to \pi_1(G(\mathcal{Y}),T)$ are the complex of groups and morphism associated to the action of $\pi_1(G(\mathcal{Y}),T)$ on $D(\mathcal{Y}, \iota_T)$ respectively.
    
    Conversely, if $G(\mathcal{Y})$ and $\iota_T$ are the complex of groups and morphism associated to the action of $\pi_1(G(\mathcal{Y}),T)$ on a simply connected scwol $\widetilde{\mathcal{Y}}$, then there is a $\pi_1(G(\mathcal{Y}),T)-$equivariant isomorphism $\widetilde{\mathcal{Y}}\to D(\mathcal{Y}, \iota_T)$ that projects to the identity on $\mathcal{Y}$.
\end{theorem}

The proposition below follows immediately from \cite[Proposition III.$\mathcal{C}$.2.18]{bh}:\begin{proposition}\cite[Proposition III.$\mathcal{C}$.2.18]{bh}
\label{inducedmapondevs}
    A morphism of complexes of groups $\phi:H(\mathcal{Y})\to G(\mathcal{X})$ over a morphism of scwols $f:\mathcal{Y}\to \mathcal{X}$ induces a $\phi_*-$equivariant morphism on developments, denoted by $\Phi:D(\mathcal{Y}, \iota_T)\to D(\mathcal{X}, \iota_{T'})$.
\end{proposition}

Our next definition lends from the topological definition of an elevation as an embedding between covers.

\begin{definition}
    If $\phi:H(\mathcal{Y})\to G(\mathcal{X})$ is a $\pi_1-$injective morphism between developable complexes of groups, we say that $$|\Phi|:|D(\mathcal{Y}, \iota_T)|\longrightarrow |D(\mathcal{X}, \iota_{T'})|$$is an \textit{elevation}.
\end{definition}

\subsection{The local development}
\label{sect_localdev} 
In this section we give explicit constructions of ``local scwols" and \textit{local developments}. These are used to define \textit{local complexes of groups} in Section \ref{sect_localCoG}. For the rest of the section we fix a scwol $\mathcal{Y}$ and $\gamma\in V(\mathcal{Y})$.

\begin{definition}\cite[Definition III.$\mathcal{C}$.1.17]{bh}
\label{def_Lk}
    The \textit{upper link of $\gamma$} is a scwol $\text{Lk}_{\gamma}$ defined by
    \begin{gather*}
        V(\text{Lk}_{\gamma})=\{c\in E(\mathcal{Y})|t(c)=\gamma\}\\ E(\text{Lk}_{\gamma})=\{(c,d)\in E^{(2)}(\mathcal{Y})|t(c)=\gamma\}
    \end{gather*}
    with source and target maps $i:(c,d)\mapsto cd, t:(c,d)\mapsto c$ and composition is defined as $$(c,d)(cd,d')\longmapsto(c,dd').$$
    
    Similarly, the \textit{lower link of $\gamma$} is a scwol $\text{Lk}^{\gamma}$ defined by
    \begin{gather*}
        V(\text{Lk}^{\gamma})=\{b\in E(\mathcal{Y})|i(b)=\gamma\}\\  E(\text{Lk}^{\gamma})=\{(a,b)\in E^{(2)}(\mathcal{Y})|i(b)=\gamma\}
    \end{gather*}
    with source and target maps $i:(a,b)\mapsto b, t:(a,b)\mapsto ab$, and composition is defined by $$(a',ab)(a,b)\longmapsto(a'a, b).$$
\end{definition}

In \cite[Definition III.$\mathcal{C}$.1.17]{bh}, the scwol $\mathcal{Y}(\gamma)$ is described as a \textit{join}\footnote{A join of scwols is defined in \cite[Definition III.$\mathcal{C}$.1.16]{bh}; its geometric realization coincides with the topological definition of a join.} of scwols consisting of $\gamma$, its upper link and its lower link. For the purpose of defining local complexes of groups in Section \ref{sect_localCoG}, we give an explicit description of $\mathcal{Y}(\gamma)$:

\begin{definition}\label{def_Y(gamma)}
The scwol $\mathcal{Y}(\gamma)$ is defined by
\begin{gather*}
    V(\mathcal{Y}(\gamma))=V(\text{Lk}_\gamma)\sqcup \{\gamma\}\sqcup V(\text{Lk}^\gamma)
\end{gather*}
\begin{align*}
    E(\mathcal{Y}(\gamma))=&E(\text{Lk}_\gamma)\\&\sqcup (\{\gamma \}\times V(\text{Lk}_\gamma))\\&\sqcup  (V(\text{Lk}_\gamma)\times  V(\text{Lk}^\gamma)) \\&\sqcup (V(\text{Lk}^\gamma)\times \{\gamma \})\\&\sqcup E(\text{Lk}^\gamma).
\end{align*}
It has source and target maps
\begin{align*}
    i:E(\mathcal{Y}(\gamma))&\longrightarrow V(\mathcal{Y}(\gamma)) &t:E(\mathcal{Y}(\gamma))&\longrightarrow V(\mathcal{Y}(\gamma))\\
    (c,d)&\longmapsto cd &(c,d)&\longmapsto c &(c,d)\in E(\text{Lk}_\gamma)\\
    \gamma*c&\longmapsto c &\gamma*c&\longmapsto \gamma &\gamma*c\in \{\gamma\}\times V(\text{Lk}_\gamma)\\
    b*c&\longmapsto c &b*c&\longmapsto b &b*c\in V(\text{Lk}_\gamma)\times  V(\text{Lk}^\gamma)\\
    b*\gamma &\longmapsto \gamma &b*\gamma &\longmapsto b & b*\gamma\in V(\text{Lk}_\gamma)\times\{\gamma\}\\
    (a,b) &\longmapsto b &(a,b) &\longmapsto ab &(a,b)\in E(\text{Lk}^\gamma)
\end{align*}
Composition on $E^{(2)}(\text{Lk}_\gamma)$ and $E^{(2)}(\text{Lk}^\gamma)$ follow from Definition \ref{def_Lk}. On the remaining pairs of composable morphisms, we have:
\begin{align*}
(\gamma*c,(c,d))&\longmapsto \gamma*cd\\
(b*c,(c,d))&\longmapsto b*cd\\
((a,b),b*c)&\longmapsto ab*c\\
(b*\gamma,\gamma*c)&\longmapsto b*c\\
(b*\gamma,\gamma*c)&\longmapsto b*c\\
((a,b),b*\gamma)&\longmapsto ab*\gamma
\end{align*}

We refer to the geometric realization of $\mathcal{Y}(\gamma)$ as the (closed) star $\text{St}(\gamma)$ of $\gamma$, while its interior is the open star $\text{st}(\gamma)$ of $\gamma$.
\end{definition}

\begin{center}
\tikzset{every picture/.style={line width=0.75pt}} 

\begin{tikzpicture}[x=0.75pt,y=0.75pt,yscale=-1,xscale=1, scale=.8]

\draw [color={rgb, 255:red, 14; green, 75; blue, 147 }  ,draw opacity=1 ]   (43.09,150.33) -- (143.89,247.22) ;
\draw [shift={(145.33,248.6)}, rotate = 223.87] [color={rgb, 255:red, 14; green, 75; blue, 147 }  ,draw opacity=1 ][line width=0.75]    (10.93,-4.9) .. controls (6.95,-2.3) and (3.31,-0.67) .. (0,0) .. controls (3.31,0.67) and (6.95,2.3) .. (10.93,4.9)   ;
\draw [color={rgb, 255:red, 14; green, 75; blue, 147 }  ,draw opacity=1 ]   (146.78,247.23) -- (250,150) ;
\draw [shift={(145.33,248.6)}, rotate = 316.71] [color={rgb, 255:red, 14; green, 75; blue, 147 }  ,draw opacity=1 ][line width=0.75]    (10.93,-4.9) .. controls (6.95,-2.3) and (3.31,-0.67) .. (0,0) .. controls (3.31,0.67) and (6.95,2.3) .. (10.93,4.9)   ;
\draw [color={rgb, 255:red, 240; green, 107; blue, 0 }  ,draw opacity=1 ]   (41.45,348.63) -- (145.33,250.4) ;
\draw [shift={(40,350)}, rotate = 316.6] [color={rgb, 255:red, 240; green, 107; blue, 0 }  ,draw opacity=1 ][line width=0.75]    (10.93,-4.9) .. controls (6.95,-2.3) and (3.31,-0.67) .. (0,0) .. controls (3.31,0.67) and (6.95,2.3) .. (10.93,4.9)   ;
\draw [color={rgb, 255:red, 240; green, 107; blue, 0 }  ,draw opacity=1 ]   (145.33,250.4) -- (248.55,348.62) ;
\draw [shift={(250,350)}, rotate = 223.58] [color={rgb, 255:red, 240; green, 107; blue, 0 }  ,draw opacity=1 ][line width=0.75]    (10.93,-3.29) .. controls (6.95,-1.4) and (3.31,-0.3) .. (0,0) .. controls (3.31,0.3) and (6.95,1.4) .. (10.93,3.29)   ;
\draw [color={rgb, 255:red, 0; green, 0; blue, 0 }  ,draw opacity=1 ]   (145.33,248.6) ;
\draw [shift={(145.33,248.6)}, rotate = 0] [color={rgb, 255:red, 0; green, 0; blue, 0 }  ,draw opacity=1 ][fill={rgb, 255:red, 0; green, 0; blue, 0 }  ,fill opacity=1 ][line width=0.75]      (0, 0) circle [x radius= 3.35, y radius= 3.35]   ;
\draw [color={rgb, 255:red, 208; green, 2; blue, 27 }  ,draw opacity=1 ]   (45.09,150.33) -- (250,150) ;
\draw [shift={(43.09,150.33)}, rotate = 359.91] [color={rgb, 255:red, 208; green, 2; blue, 27 }  ,draw opacity=1 ][line width=0.75]    (10.93,-4.9) .. controls (6.95,-2.3) and (3.31,-0.67) .. (0,0) .. controls (3.31,0.67) and (6.95,2.3) .. (10.93,4.9)   ;
\draw [color={rgb, 255:red, 65; green, 117; blue, 5 }  ,draw opacity=1 ]   (40,350) -- (248,350) ;
\draw [shift={(250,350)}, rotate = 180] [color={rgb, 255:red, 65; green, 117; blue, 5 }  ,draw opacity=1 ][line width=0.75]    (10.93,-4.9) .. controls (6.95,-2.3) and (3.31,-0.67) .. (0,0) .. controls (3.31,0.67) and (6.95,2.3) .. (10.93,4.9)   ;
\draw [color={rgb, 255:red, 128; green, 128; blue, 128 }  ,draw opacity=1 ]   (41.07,150.34) -- (40.01,348) ;
\draw [shift={(40,350)}, rotate = 270.31] [color={rgb, 255:red, 128; green, 128; blue, 128 }  ,draw opacity=1 ][line width=0.75]    (10.93,-3.29) .. controls (6.95,-1.4) and (3.31,-0.3) .. (0,0) .. controls (3.31,0.3) and (6.95,1.4) .. (10.93,3.29)   ;
\draw [color={rgb, 255:red, 128; green, 128; blue, 128 }  ,draw opacity=1 ]   (250,150) -- (250,348) ;
\draw [shift={(250,350)}, rotate = 270] [color={rgb, 255:red, 128; green, 128; blue, 128 }  ,draw opacity=1 ][line width=0.75]    (10.93,-3.29) .. controls (6.95,-1.4) and (3.31,-0.3) .. (0,0) .. controls (3.31,0.3) and (6.95,1.4) .. (10.93,3.29)   ;
\draw [color={rgb, 255:red, 155; green, 155; blue, 155 }  ,draw opacity=1 ]   (43.09,150.33) .. controls (125.71,177.52) and (207.17,260.98) .. (249.37,348.68) ;
\draw [shift={(250,350)}, rotate = 244.53] [color={rgb, 255:red, 155; green, 155; blue, 155 }  ,draw opacity=1 ][line width=0.75]    (10.93,-3.29) .. controls (6.95,-1.4) and (3.31,-0.3) .. (0,0) .. controls (3.31,0.3) and (6.95,1.4) .. (10.93,3.29)   ;
\draw [color={rgb, 255:red, 155; green, 155; blue, 155 }  ,draw opacity=1 ]   (250,150) .. controls (223,162.4) and (189.67,177.17) .. (167.78,192.49) .. controls (146.01,207.73) and (74.37,278.59) .. (40.51,348.94) ;
\draw [shift={(40,350)}, rotate = 295.45] [color={rgb, 255:red, 155; green, 155; blue, 155 }  ,draw opacity=1 ][line width=0.75]    (10.93,-3.29) .. controls (6.95,-1.4) and (3.31,-0.3) .. (0,0) .. controls (3.31,0.3) and (6.95,1.4) .. (10.93,3.29)   ;
\draw [color={rgb, 255:red, 208; green, 2; blue, 27 }  ,draw opacity=1 ]   (43.09,150.33) ;
\draw [shift={(43.09,150.33)}, rotate = 0] [color={rgb, 255:red, 208; green, 2; blue, 27 }  ,draw opacity=1 ][fill={rgb, 255:red, 208; green, 2; blue, 27 }  ,fill opacity=1 ][line width=0.75]      (0, 0) circle [x radius= 3.35, y radius= 3.35]   ;
\draw [color={rgb, 255:red, 208; green, 2; blue, 27 }  ,draw opacity=1 ]   (250,150) ;
\draw [shift={(250,150)}, rotate = 0] [color={rgb, 255:red, 208; green, 2; blue, 27 }  ,draw opacity=1 ][fill={rgb, 255:red, 208; green, 2; blue, 27 }  ,fill opacity=1 ][line width=0.75]      (0, 0) circle [x radius= 3.35, y radius= 3.35]   ;
\draw [color={rgb, 255:red, 65; green, 117; blue, 5 }  ,draw opacity=1 ]   (40,350) ;
\draw [shift={(40,350)}, rotate = 0] [color={rgb, 255:red, 65; green, 117; blue, 5 }  ,draw opacity=1 ][fill={rgb, 255:red, 65; green, 117; blue, 5 }  ,fill opacity=1 ][line width=0.75]      (0, 0) circle [x radius= 3.35, y radius= 3.35]   ;
\draw [color={rgb, 255:red, 65; green, 117; blue, 5 }  ,draw opacity=1 ]   (250,350) ;
\draw [shift={(250,350)}, rotate = 0] [color={rgb, 255:red, 65; green, 117; blue, 5 }  ,draw opacity=1 ][fill={rgb, 255:red, 65; green, 117; blue, 5 }  ,fill opacity=1 ][line width=0.75]      (0, 0) circle [x radius= 3.35, y radius= 3.35]   ;
\draw [color={rgb, 255:red, 14; green, 75; blue, 147 }  ,draw opacity=1 ]   (384.09,150.33) -- (484.89,247.22) ;
\draw [shift={(486.33,248.6)}, rotate = 223.87] [color={rgb, 255:red, 14; green, 75; blue, 147 }  ,draw opacity=1 ][line width=0.75]    (10.93,-4.9) .. controls (6.95,-2.3) and (3.31,-0.67) .. (0,0) .. controls (3.31,0.67) and (6.95,2.3) .. (10.93,4.9)   ;
\draw [color={rgb, 255:red, 14; green, 75; blue, 147 }  ,draw opacity=1 ]   (487.78,247.23) -- (591,150) ;
\draw [shift={(486.33,248.6)}, rotate = 316.71] [color={rgb, 255:red, 14; green, 75; blue, 147 }  ,draw opacity=1 ][line width=0.75]    (10.93,-4.9) .. controls (6.95,-2.3) and (3.31,-0.67) .. (0,0) .. controls (3.31,0.67) and (6.95,2.3) .. (10.93,4.9)   ;
\draw [color={rgb, 255:red, 240; green, 107; blue, 0 }  ,draw opacity=1 ]   (382.45,348.63) -- (486.33,250.4) ;
\draw [shift={(381,350)}, rotate = 316.6] [color={rgb, 255:red, 240; green, 107; blue, 0 }  ,draw opacity=1 ][line width=0.75]    (10.93,-4.9) .. controls (6.95,-2.3) and (3.31,-0.67) .. (0,0) .. controls (3.31,0.67) and (6.95,2.3) .. (10.93,4.9)   ;
\draw [color={rgb, 255:red, 240; green, 107; blue, 0 }  ,draw opacity=1 ]   (486.33,250.4) -- (589.55,348.62) ;
\draw [shift={(591,350)}, rotate = 223.58] [color={rgb, 255:red, 240; green, 107; blue, 0 }  ,draw opacity=1 ][line width=0.75]    (10.93,-3.29) .. controls (6.95,-1.4) and (3.31,-0.3) .. (0,0) .. controls (3.31,0.3) and (6.95,1.4) .. (10.93,3.29)   ;
\draw [color={rgb, 255:red, 0; green, 0; blue, 0 }  ,draw opacity=1 ]   (486.33,248.6) ;
\draw [shift={(486.33,248.6)}, rotate = 0] [color={rgb, 255:red, 0; green, 0; blue, 0 }  ,draw opacity=1 ][fill={rgb, 255:red, 0; green, 0; blue, 0 }  ,fill opacity=1 ][line width=0.75]      (0, 0) circle [x radius= 3.35, y radius= 3.35]   ;
\draw [color={rgb, 255:red, 208; green, 2; blue, 27 }  ,draw opacity=1 ]   (386.09,150.33) -- (591,150) ;
\draw [shift={(384.09,150.33)}, rotate = 359.91] [color={rgb, 255:red, 208; green, 2; blue, 27 }  ,draw opacity=1 ][line width=0.75]    (10.93,-4.9) .. controls (6.95,-2.3) and (3.31,-0.67) .. (0,0) .. controls (3.31,0.67) and (6.95,2.3) .. (10.93,4.9)   ;
\draw [color={rgb, 255:red, 65; green, 117; blue, 5 }  ,draw opacity=1 ]   (381,350) -- (589,350) ;
\draw [shift={(591,350)}, rotate = 180] [color={rgb, 255:red, 65; green, 117; blue, 5 }  ,draw opacity=1 ][line width=0.75]    (10.93,-4.9) .. controls (6.95,-2.3) and (3.31,-0.67) .. (0,0) .. controls (3.31,0.67) and (6.95,2.3) .. (10.93,4.9)   ;
\draw [color={rgb, 255:red, 128; green, 128; blue, 128 }  ,draw opacity=1 ]   (382.07,150.34) -- (381.01,348) ;
\draw [shift={(381,350)}, rotate = 270.31] [color={rgb, 255:red, 128; green, 128; blue, 128 }  ,draw opacity=1 ][line width=0.75]    (10.93,-3.29) .. controls (6.95,-1.4) and (3.31,-0.3) .. (0,0) .. controls (3.31,0.3) and (6.95,1.4) .. (10.93,3.29)   ;
\draw [color={rgb, 255:red, 128; green, 128; blue, 128 }  ,draw opacity=1 ]   (591,150) -- (591,348) ;
\draw [shift={(591,350)}, rotate = 270] [color={rgb, 255:red, 128; green, 128; blue, 128 }  ,draw opacity=1 ][line width=0.75]    (10.93,-3.29) .. controls (6.95,-1.4) and (3.31,-0.3) .. (0,0) .. controls (3.31,0.3) and (6.95,1.4) .. (10.93,3.29)   ;
\draw [color={rgb, 255:red, 155; green, 155; blue, 155 }  ,draw opacity=1 ]   (384.09,150.33) .. controls (466.71,177.52) and (548.17,260.98) .. (590.37,348.68) ;
\draw [shift={(591,350)}, rotate = 244.53] [color={rgb, 255:red, 155; green, 155; blue, 155 }  ,draw opacity=1 ][line width=0.75]    (10.93,-3.29) .. controls (6.95,-1.4) and (3.31,-0.3) .. (0,0) .. controls (3.31,0.3) and (6.95,1.4) .. (10.93,3.29)   ;
\draw [color={rgb, 255:red, 155; green, 155; blue, 155 }  ,draw opacity=1 ]   (591,150) .. controls (564,162.4) and (530.67,177.17) .. (508.78,192.49) .. controls (487.01,207.73) and (415.37,278.59) .. (381.51,348.94) ;
\draw [shift={(381,350)}, rotate = 295.45] [color={rgb, 255:red, 155; green, 155; blue, 155 }  ,draw opacity=1 ][line width=0.75]    (10.93,-3.29) .. controls (6.95,-1.4) and (3.31,-0.3) .. (0,0) .. controls (3.31,0.3) and (6.95,1.4) .. (10.93,3.29)   ;
\draw [color={rgb, 255:red, 14; green, 75; blue, 147 }  ,draw opacity=1 ]   (384.09,150.33) ;
\draw [shift={(384.09,150.33)}, rotate = 0] [color={rgb, 255:red, 14; green, 75; blue, 147 }  ,draw opacity=1 ][fill={rgb, 255:red, 14; green, 75; blue, 147 }  ,fill opacity=1 ][line width=0.75]      (0, 0) circle [x radius= 3.35, y radius= 3.35]   ;
\draw [color={rgb, 255:red, 14; green, 75; blue, 147 }  ,draw opacity=1 ]   (591,150) ;
\draw [shift={(591,150)}, rotate = 0] [color={rgb, 255:red, 14; green, 75; blue, 147 }  ,draw opacity=1 ][fill={rgb, 255:red, 14; green, 75; blue, 147 }  ,fill opacity=1 ][line width=0.75]      (0, 0) circle [x radius= 3.35, y radius= 3.35]   ;
\draw [color={rgb, 255:red, 240; green, 107; blue, 0 }  ,draw opacity=1 ]   (381,350) ;
\draw [shift={(381,350)}, rotate = 0] [color={rgb, 255:red, 240; green, 107; blue, 0 }  ,draw opacity=1 ][fill={rgb, 255:red, 240; green, 107; blue, 0 }  ,fill opacity=1 ][line width=0.75]      (0, 0) circle [x radius= 3.35, y radius= 3.35]   ;
\draw [color={rgb, 255:red, 240; green, 107; blue, 0 }  ,draw opacity=1 ]   (591,350) ;
\draw [shift={(591,350)}, rotate = 0] [color={rgb, 255:red, 240; green, 107; blue, 0 }  ,draw opacity=1 ][fill={rgb, 255:red, 240; green, 107; blue, 0 }  ,fill opacity=1 ][line width=0.75]      (0, 0) circle [x radius= 3.35, y radius= 3.35]   ;

\draw (30,110) node [anchor=north west][inner sep=0.75pt]    {$\text{Minimal Subscwol of } \mathcal{Y} \text{ Containing }\gamma$};
\draw (140.78,131.51) node [anchor=north west][inner sep=0.75pt]  [color={rgb, 255:red, 208; green, 2; blue, 27 }  ,opacity=1 ,rotate=-0.59]  {$d$};
\draw (136.45,331.31) node [anchor=north west][inner sep=0.75pt]  [color={rgb, 255:red, 65; green, 117; blue, 5 }  ,opacity=1 ,rotate=-0.59]  {$a$};
\draw (93.15,213.68) node [anchor=north west][inner sep=0.75pt]  [color={rgb, 255:red, 14; green, 75; blue, 147 }  ,opacity=1 ,rotate=-0.59]  {$c$};
\draw (188.43,209.93) node [anchor=north west][inner sep=0.75pt]  [color={rgb, 255:red, 14; green, 75; blue, 147 }  ,opacity=1 ,rotate=-0.59]  {$cd$};
\draw (94.63,303.62) node [anchor=north west][inner sep=0.75pt]  [color={rgb, 255:red, 240; green, 107; blue, 0 }  ,opacity=1 ,rotate=-0.59]  {$b$};
\draw (44.62,243.31) node [anchor=north west][inner sep=0.75pt]  [color={rgb, 255:red, 128; green, 128; blue, 128 }  ,opacity=1 ]  {$bc$};
\draw (179.32,301.78) node [anchor=north west][inner sep=0.75pt]  [color={rgb, 255:red, 240; green, 107; blue, 0 }  ,opacity=1 ,rotate=-0.59]  {$ab$};
\draw (138.47,261.95) node [anchor=north west][inner sep=0.75pt]  [font=\large,color={rgb, 255:red, 0; green, 0; blue, 0 }  ,opacity=1 ,rotate=-0.59]  {$\gamma $};
\draw (462.63,110) node [anchor=north west][inner sep=0.75pt]    {$\mathcal{Y}( \gamma )$};
\draw (378.37,128.31) node [anchor=north west][inner sep=0.75pt]  [color={rgb, 255:red, 14; green, 75; blue, 147 }  ,opacity=1 ,rotate=-0.59]  {$c$};
\draw (583.56,127.64) node [anchor=north west][inner sep=0.75pt]  [color={rgb, 255:red, 14; green, 75; blue, 147 }  ,opacity=1 ,rotate=-0.59]  {$cd$};
\draw (374.47,355.95) node [anchor=north west][inner sep=0.75pt]  [color={rgb, 255:red, 65; green, 117; blue, 5 }  ,opacity=1 ,rotate=-0.59]  {$b$};
\draw (462.78,131.51) node [anchor=north west][inner sep=0.75pt]  [color={rgb, 255:red, 208; green, 2; blue, 27 }  ,opacity=1 ,rotate=-0.59]  {$( c,d)$};
\draw (467.45,331.31) node [anchor=north west][inner sep=0.75pt]  [color={rgb, 255:red, 65; green, 117; blue, 5 }  ,opacity=1 ,rotate=-0.59]  {$( a,b)$};
\draw (408,212) node [anchor=north west][inner sep=0.75pt]  [color={rgb, 255:red, 14; green, 75; blue, 147 }  ,opacity=1 ,rotate=-0.59]  {$\gamma *c$};
\draw (534.43,206.93) node [anchor=north west][inner sep=0.75pt]  [color={rgb, 255:red, 14; green, 75; blue, 147 }  ,opacity=1 ,rotate=-0.59]  {$\gamma *cd$};
\draw (435.63,303.62) node [anchor=north west][inner sep=0.75pt]  [color={rgb, 255:red, 240; green, 107; blue, 0 }  ,opacity=1 ,rotate=-0.59]  {$b*\gamma $};
\draw (385.62,242.31) node [anchor=north west][inner sep=0.75pt]  [color={rgb, 255:red, 128; green, 128; blue, 128 }  ,opacity=1 ]  {$b*c$};
\draw (502.32,301.78) node [anchor=north west][inner sep=0.75pt]  [color={rgb, 255:red, 240; green, 107; blue, 0 }  ,opacity=1 ,rotate=-0.59]  {$ab*\gamma $};
\draw (585.47,355.95) node [anchor=north west][inner sep=0.75pt]  [color={rgb, 255:red, 65; green, 117; blue, 5 }  ,opacity=1 ,rotate=-0.59]  {$ab$};
\draw (479.47,263.95) node [anchor=north west][inner sep=0.75pt]  [font=\large,color={rgb, 255:red, 0; green, 0; blue, 0 }  ,opacity=1 ,rotate=-0.59]  {$\gamma $};

\end{tikzpicture}

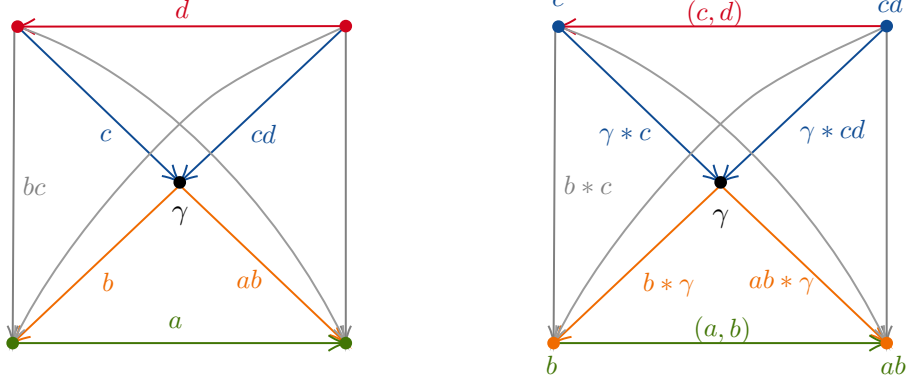
\captionof{figure}{The scwol $\mathcal{Y}(\gamma)$ (right) illustrated by all the types of morphisms it contains.}
\end{center}

\begin{proposition}\cite[III.$\mathcal{C}$.1.17, p.533]{bh}\label{prop_morphism_Y(gamma)toY}
    The morphism of scwols $h:\mathcal{Y}(\gamma)\to\mathcal{Y}$ is defined by
\begin{align*}
    V(\mathcal{Y}(\gamma))&\longrightarrow V(\mathcal{Y})\\
    c&\longmapsto i(c) &c\in V(\text{Lk}_\gamma)\\
    \gamma&\longmapsto\gamma &\gamma\in\{\gamma\}\\
    b&\longmapsto t(b) &b\in V(\text{Lk}^\gamma)
\end{align*}
\begin{align*}
    E(\mathcal{Y}(\gamma))&\longrightarrow E(\mathcal{Y})\\
    (c,d)&\longmapsto d &(c,d)\in E(\text{Lk}_\gamma)\\
    \gamma*c&\longmapsto c &\gamma*c\in \{\gamma\}\times V(\text{Lk}_\gamma)\\
    b*c&\longmapsto bc &b*c\in V(\text{Lk}^\gamma)\times V(\text{Lk}_\gamma)\\
    b*\gamma&\longmapsto b &b*\gamma\in V(\text{Lk}^\gamma)\times\{\gamma\}\\
    (a,b)&\longmapsto a &(a,b)\in E(\text{Lk}^\gamma)
\end{align*}
\end{proposition}

In the next definition, we use $gG_{i(c)}\in \rightQ{G_\gamma}{G_{i(c)}}$ to express the coset $g\psi_c(G_{i(c)})
\in \rightQ{G_\gamma}{\psi_c(G_{i(c)})}$.

\begin{definition}\cite[Definition 4.20]{bh}\label{def_upperlinkdev}
    Let $\text{Lk}_{\tilde{\gamma}}$ be a scwol defined by
    \begin{gather*}
        V(\text{Lk}_{\tilde{\gamma}})=\{(gG_{i(c)},c)| t(c)=\gamma, gG_{i(c)}\in \rightQ{G_\gamma}{G_{i(c)}}\}
        \\E(\text{Lk}_{\tilde{\gamma}})=\{(gG_{i(d)},c,d)|(c,d)\in E^{(2)}(\mathcal{Y}), t(c)=\gamma,gG_{i(d)}\in \rightQ{G_\gamma}{G_{i(d)}}\},
    \end{gather*}
    where
    $
        i:(gG_{i(d)},c,d)\mapsto (gG_{i(d)},cd)\text{ and }
        t:(gG_{i(d)},c,d)\mapsto (g{g_{c,d}}^{-1}G_{i(c)},c)
    $, and composition is defined by
    $$(gG_{i(d)}, c,d)(gg_{cd,d'}G_{i(d')},cd,d')\mapsto(gg_{cd,d'}G_{i(d')},c,dd').$$
\end{definition}

Similarly, the scwol $\mathcal{Y}(\tilde{\gamma})$ is defined in \cite[Definition III.$\mathcal{C}$.4.21]{bh} as the join of $\text{Lk}^\gamma, \gamma$ and $\text{Lk}_{\tilde{\gamma}}$. Using these definitions, we obtain an explicit construction of $\mathcal{Y}(\tilde{\gamma})$:

\begin{definition}\label{def_localdevelopment}
    Let $G(\mathcal{Y})$ be a complex of groups over $\mathcal{Y}$.  The \textit{local development of $\gamma$}, denoted by $\mathcal{Y}(\tilde{\gamma})$, is the scwol defined by
    \begin{gather*}
    V(\mathcal{Y}(\tilde{\gamma}))=V(\text{Lk}_{\tilde{\gamma}})\sqcup \{\gamma\}\sqcup V(\text{Lk}^\gamma)\\
    E(\mathcal{Y}(\tilde{\gamma}))=E(\text{Lk}_{\tilde{\gamma}}) 
        \sqcup (\{\gamma\}\times V(\text{Lk}_{\tilde{\gamma}}))
        \sqcup (V(\text{Lk}^\gamma)\times V(\text{Lk}_{\tilde{\gamma}}))
        \sqcup (V(\text{Lk}^{\gamma})\times \{\gamma\})
        \sqcup E(\text{Lk}^{\gamma}),
    \end{gather*}
    where
    \begin{align*}
    i:E(\mathcal{Y}(\tilde{\gamma}))&\longrightarrow V(\mathcal{Y}(\tilde{\gamma}))\\
    (gG_{i(d)},c,d)&\longmapsto (gG_{i(d)},cd)
    &
    (gG_{i(d)},c,d)\in E(\text{Lk}_{\tilde{\gamma}})
    \\
    \gamma*(gG_{i(c)},c)&\longmapsto (gG_{i(c)},c)
    &
    \gamma*(gG_{i(c)},c)\in \{\gamma\}\times V(\text{Lk}_{\tilde{\gamma}})
    \\b*(gG_{i(c)},c)&\longmapsto (gG_{i(c)},c)
    &
    b*(gG_{i(c)},c)\in V(\text{Lk}^\gamma)\times V(\text{Lk}_{\tilde{\gamma}})
    \\b*\gamma&\longmapsto \gamma
    &
    b*\gamma\in V(\text{Lk}^\gamma)\times\{\gamma\}
    \\
    (a,b)&\longmapsto b
    &(a,b)\in E(\text{Lk}^\gamma)
    \end{align*}
and
\begin{align*}
    t:E(\mathcal{Y}(\tilde{\gamma}))&\longrightarrow V(\mathcal{Y}(\tilde{\gamma}))\\
    (gG_{i(d)},c,d)&\longmapsto (g{g_{c,d}}^{-1}G_{i(c)},c) &
    (gG_{i(d)},c,d)\in E(\text{Lk}_{\tilde{\gamma}})\\
    \gamma*(gG_{i(c)},c)&\longmapsto \gamma &\gamma*(gG_{i(c)},c)\in \{\gamma\}\times V(\text{Lk}_{\tilde{\gamma}})\\
    b*(gG_{i(c)},c)&\longmapsto b &b*(gG_{i(c)},c)\in V(\text{Lk}^\gamma)\times V(\text{Lk}_{\tilde{\gamma}})\\
    b*\gamma&\longmapsto b &b*\gamma\in V(\text{Lk}^\gamma)\times\{\gamma\}\\
    (a,b)&\longmapsto & ab(a,b)\in E(\text{Lk}^\gamma)
\end{align*}
    
The geometric realization of $\mathcal{Y}(\tilde{\gamma})$ is denoted by $\text{St}(\tilde{\gamma})$, while its interior is denoted by $\text{st}(\tilde{\gamma})$.

Composition on $E^{(2)}(\text{Lk}_{\tilde{\gamma}})$ and $E^{(2)}(\text{Lk}^\gamma)$ follow from Definition \ref{def_Lk}. On the remaining pairs of composable morphisms, we have:
\begin{align*}
(\gamma*(gG_{i(c)},c),(gg_{c,d}G_{i(d)},c,d))&\longmapsto \gamma*(gg_{c,d}G_{i(d)},cd)\\
(b*(gG_{i(c)},c),(gg_{c,d}G_{i(d)},c,d))&\longmapsto b*(gg_{c,d}G_{i(d)},cd)\\
((a,b),b*(gG_{i(c)},c))&\longmapsto ab*(gG_{i(c)},c)\\
(b*\gamma,\gamma*(gG_{i(c)},c))&\longmapsto b*(gG_{i(c)},c)\\
(b*\gamma,\gamma*(gG_{i(c)},c))&\longmapsto b*(gG_{i(c)},c)\\
((a,b),b*\gamma)&\longmapsto ab*\gamma
\end{align*}
\end{definition}

The next proposition is an immediate consequence of \cite[Proposition III.$\mathcal{C}$.4.11]{bh}.\begin{proposition}\label{prop_embeddingofstars}
    Let $G(\mathcal{Y})$ be a developable complex of groups over a scwol $\mathcal{Y}$.
    For each lift $\bar{\gamma}$ of $\gamma\in V(\mathcal{Y})$ in $D(\mathcal{Y},\iota_T)$, there exists a $\text{Stab}(\bar{\gamma})-$equivariant embedding given by
    \begin{gather*}
        \text{st}(\tilde{\gamma})\longrightarrow |D(\mathcal{Y},\iota_T)|\\
        \text{st}(\tilde{\gamma})\longmapsto \text{st}(\bar{\gamma}).
    \end{gather*}
\end{proposition}

\begin{proposition}\cite[Proposition III.$\mathcal{C}$.4.23]{bh}\label{prop_morphismlocaldevelopments}
    A morphism of complexes of groups $$\phi:H(\mathcal{Y})\longrightarrow G(\mathcal{X})$$ over a $\phi_\sigma-$equivariant morphism of scwols $f:\mathcal{Y}\to \mathcal{X}$ induces, for each $\sigma\in V(\mathcal{Y})$, a morphism 
    $$
        \Phi_\sigma: \mathcal{Y}(\tilde{\sigma})\longrightarrow \mathcal{X}(\widetilde{f(\sigma)})$$
    of local developments.
\end{proposition} 
In particular, on $\text{Lk}_{\tilde{\sigma}}$, the morphism is described as:
\begin{align*}
    (h\psi_c(H_{i(c)}),c)&\longmapsto (\phi_{\sigma}(h)\phi(c)(G_{i(f(c))}), f(c)) \\
    (h\psi_{cd}(H_{i(cd)}),c, d)&\longmapsto (\phi_{\sigma}(h)\phi(cd)(G_{i(f(cd))}), f(c), f(d))
\end{align*}
where $(h\psi_c(H_{i(c)}),c)\in V(\text{Lk}_{\tilde{\sigma}})$ and  $(h\psi_{cd}(H_{i(cd)}),c, d) \in E(\text{Lk}_{\tilde{\sigma}})$.

\subsection{Metrizing a complex of groups}\label{section_metrize}
We recall from \cite{bh} the process of metrizing a complex of groups and an important developability theorem for non-positively curved complexes of groups.

\begin{definition}\cite[Definition I.7.37]{bh}
    We \textit{metrize} a polyhedral complex $K$ by taking,
    for each $n-$simplex $\triangle$ of $K$, a geodesic $n-$simplex in a model space of curvature $\leq \kappa$ such that: 
\begin{enumerate}
    \item the interior of each face of $\triangle$ maps injectively into $K$, and
    \item any other geodesic $n-$simplex intersecting $\triangle$ in $K$ non-trivially (on a unique face) is isometric to $\triangle$ on this face.
\end{enumerate}
The metrized polyhedral complex is said to be an \textit{$M_\kappa-$polyhedral complex}.
\end{definition}

Let $\mathcal{Y}$ be a scwol. Then its geometric realization $|\mathcal{Y}|$ is a polyhedral complex. For such polyhedral complexes, we impose the following assumption:
\begin{assumption}\label{assumption_Shapes} On the geometric realization $|\mathcal{Y}|$ of a scwol $\mathcal{Y}$,
the set of isometry classes Shapes($|\mathcal{Y}|$) of the faces of geodesic simplices is finite.\end{assumption}

\begin{remark}\label{remark_metrized_morphism}Let $f:\mathcal{Y}\to \mathcal{X}$ be a non-degenerate morphism of scwols and $|\mathcal{X}|$ be an $M_\kappa-$polyhedral complex. Recall by non-degeneracy of $f$ that $|f|: |\mathcal{Y}| \to |\mathcal{X}|$ is an affine homeomorphism on the interiors of cells of $|\mathcal{Y}|$. By defining the metric on $|\mathcal{Y}|$ as the pullback of the metric on $|\mathcal{X}|$, we say that the map $|f|: |\mathcal{Y}| \to |\mathcal{X}|$ is a local isometry on the interiors of cells, hence, $|\mathcal{Y}|$ is an $M_\kappa-$polyhedral complex.
\end{remark}

Let $\sigma\in V(\mathcal{Y})$ and let $G(\mathcal{Y})$ be a scwol over $\mathcal{Y}$. By Remark \ref{remark_metrized_morphism}, the maps
\begin{gather*}
    |f_\sigma|:\text{St}(\sigma)\longrightarrow |\mathcal{Y}| \text{ and } \text{St}(\tilde{\sigma}) \longrightarrow \text{St}(\sigma)
\end{gather*}
are local isometries on interiors of cells. Thus, finiteness of $\text{Shapes}(|\mathcal{Y}|)$ implies finiteness of $\text{Shapes}(\text{St}(\tilde{\sigma}))$ for all $\sigma\in v(\mathcal{Y})$.

\begin{definition}\cite[Definition III.\(\mathcal{C}\).4.16]{bh}\label{def_npcCoG}
    A complex of groups $G(\mathcal{Y})$ is \textit{non-positively curved} if for all $\sigma\in V(\mathcal{Y})$, $\text{st}(\tilde{\sigma})$ is non-positively curved.
\end{definition}

\begin{theorem}\cite[Theorem III.$\mathcal{C}$.4.17]{bh}
\label{thm_NPC}
A non-positively curved complex of groups is developable.
\end{theorem}

This result was proven in \cite[III.$\mathcal{G}$]{bh} using the language of groupoids of local isometries. It is the analog of the Cartan--Hadamard theorem for complexes of groups. In particular, by \cite[Proposition III.$\mathcal{C}$.4.1]{bh}, $\text{St}(\tilde{\gamma})$ embeds into $|D(\mathcal{Y},\iota_T)|$ for all $\gamma\in V(\mathcal{Y})$. Hence, $|D(\mathcal{Y},\iota_T)|$ is non-positively curved. Moreover, Theorem \ref{thm_D(Y,iota_T)} tells us that $|D(\mathcal{Y},\iota_T)|$ is simply connected, hence, it is CAT(0).

\section{Local complexes of groups}
\label{sect_localCoG}
In this section, we define a \textit{local complex of groups} and offer a functorial perspective of the statement ``every complex of groups is locally developable" \cite[III.$\mathcal{C}$, p.520]{bh}. Our approach is to construct a complex of groups from the local data of a given complex of groups (see Definition \ref{def_localCoG} and Proposition \ref{prop_localisCoG}). In Theorem \ref{thm_localCoG}, we show that a local complex of groups reflects the local data of a complex of groups. In particular, a local complex of groups is always developable and its development is isomorphic to the local development\footnote{We refer the reader to Definition \ref{def_localdevelopment} for the construction of a local development.}. This is the content of Propositions \ref{prop_localdevelopability} and Proposition \ref{prop_development_localCoG}. 
Lastly, we prove a functorial version of the developability criterion in Corollary \ref{cor:thm_localCoG}. 

For the rest of this section, we fix a complex of groups $G(\mathcal{Y})$ over a scwol $\mathcal{Y}$ and $\gamma\in V(\mathcal{Y})$. Recall the definition of the scwol $\mathcal{Y}(\gamma)$ from Definition \ref{def_Y(gamma)}.\begin{definition}\label{def_localCoG}
    Given a complex of groups $G(\mathcal{Y})=(G_\sigma,\psi_a, g_{a,b})$ over $\mathcal{Y}$, the \textit{local complex of group over $\gamma$}, $L(\mathcal{Y}(\gamma))$, is the following set of data over $\mathcal{Y}(\gamma)$:\begin{enumerate}
        \item Over each object in $V(\mathcal{Y}(\gamma))$, local groups
        
        $L_-=\begin{cases}
            G_{i(c)} &c\in V(\text{Lk}_\gamma)\\
            G_\gamma &\gamma\in \{\gamma\}\\
            G_\gamma &b\in V(\text{Lk}^\gamma)\\
        \end{cases}$
        \item Over each morphism in $E(\mathcal{Y}(\gamma))$, injective homomorphisms
        
        $\lambda_-=\begin{cases}
            \psi_d &(c,d)\in E(\text{Lk}_\gamma)\\
            \psi_c &\gamma*c\in \{\gamma\}\times V(\text{Lk}_\gamma)\\
            \psi_c &b*c\in V(\text{Lk}^\gamma)\times V(\text{Lk}_\gamma)\\
            \text{id}_{G_\gamma} &b*\gamma\in V(\text{Lk}^\gamma)\times\{\gamma\}\\
            \text{id}_{G_\gamma} &(a,b)\in E(\text{Lk}^\gamma)
        \end{cases}$
        \item Associated to each pair of morphisms in $E^{(2)}(\mathcal{Y}(\gamma))$, twisting elements
        
        $l_{-,-}=\begin{cases}
            g_{d_1,d_2} &((c_1,d_1),(c_2,d_2))\in E^{(2)}(\text{Lk}_\gamma)\\
            g_{c,d} &(\gamma*c, (c,d))\in (\{\gamma\}\times V(\text{Lk}_\gamma))\times E(\text{Lk}_\gamma)\\
            g_{c,d} &(b*c,(c,d))\in (V(\text{Lk}^\gamma)\times V(\text{Lk}_\gamma))\times E(\text{Lk}_\gamma)\\
            e &((a,b),b*c)\in E(\text{Lk}^\gamma)\times(V(\text{Lk}^\gamma)\times V(\text{Lk}_\gamma)) \\
            e &(b*\gamma, \gamma*c)\in (V(\text{Lk}^\gamma)\times \{\gamma\})\times ( \{\gamma\}\times V(\text{Lk}_\gamma))\\
            e &((a,b),b*\gamma)\in E(\text{Lk}^\gamma*\gamma)\times (V(\text{Lk}^\gamma)\times \{\gamma\})\\
            e &((a_1,b_1),(a_2,b_2))\in E^{(2)}(\text{Lk}_\gamma)
        \end{cases}$
    \end{enumerate}
\end{definition}
The next figure illustrates a local complex of groups using the local groups and injective homomorphisms of $G(\mathcal{Y})$.
\begin{center}
\tikzset{every picture/.style={line width=0.75pt}} 

\begin{tikzpicture}[x=0.75pt,y=0.75pt,yscale=-1,xscale=1, scale = 1.2]

\draw [color={rgb, 255:red, 14; green, 75; blue, 147 }  ,draw opacity=1 ]   (215.09,163.33) -- (315.89,260.22) ;
\draw [shift={(317.33,261.6)}, rotate = 223.87] [color={rgb, 255:red, 14; green, 75; blue, 147 }  ,draw opacity=1 ][line width=0.75]    (10.93,-4.9) .. controls (6.95,-2.3) and (3.31,-0.67) .. (0,0) .. controls (3.31,0.67) and (6.95,2.3) .. (10.93,4.9)   ;
\draw [color={rgb, 255:red, 14; green, 75; blue, 147 }  ,draw opacity=1 ]   (318.77,260.22) -- (419.27,163.67) ;
\draw [shift={(317.33,261.6)}, rotate = 316.15] [color={rgb, 255:red, 14; green, 75; blue, 147 }  ,draw opacity=1 ][line width=0.75]    (10.93,-4.9) .. controls (6.95,-2.3) and (3.31,-0.67) .. (0,0) .. controls (3.31,0.67) and (6.95,2.3) .. (10.93,4.9)   ;
\draw [color={rgb, 255:red, 240; green, 107; blue, 0 }  ,draw opacity=1 ]   (216.83,359.94) -- (317.33,263.4) ;
\draw [shift={(215.39,361.33)}, rotate = 316.15] [color={rgb, 255:red, 240; green, 107; blue, 0 }  ,draw opacity=1 ][line width=0.75]    (10.93,-4.9) .. controls (6.95,-2.3) and (3.31,-0.67) .. (0,0) .. controls (3.31,0.67) and (6.95,2.3) .. (10.93,4.9)   ;
\draw [color={rgb, 255:red, 240; green, 107; blue, 0 }  ,draw opacity=1 ]   (317.33,263.4) -- (418.13,360.29) ;
\draw [shift={(419.57,361.67)}, rotate = 223.87] [color={rgb, 255:red, 240; green, 107; blue, 0 }  ,draw opacity=1 ][line width=0.75]    (10.93,-3.29) .. controls (6.95,-1.4) and (3.31,-0.3) .. (0,0) .. controls (3.31,0.3) and (6.95,1.4) .. (10.93,3.29)   ;
\draw [color={rgb, 255:red, 0; green, 0; blue, 0 }  ,draw opacity=1 ]   (317.33,261.6) ;
\draw [shift={(317.33,261.6)}, rotate = 0] [color={rgb, 255:red, 0; green, 0; blue, 0 }  ,draw opacity=1 ][fill={rgb, 255:red, 0; green, 0; blue, 0 }  ,fill opacity=1 ][line width=0.75]      (0, 0) circle [x radius= 3.35, y radius= 3.35]   ;
\draw [color={rgb, 255:red, 208; green, 2; blue, 27 }  ,draw opacity=1 ]   (217.09,163.33) -- (419.27,163.67) ;
\draw [shift={(215.09,163.33)}, rotate = 0.1] [color={rgb, 255:red, 208; green, 2; blue, 27 }  ,draw opacity=1 ][line width=0.75]    (10.93,-4.9) .. controls (6.95,-2.3) and (3.31,-0.67) .. (0,0) .. controls (3.31,0.67) and (6.95,2.3) .. (10.93,4.9)   ;
\draw [color={rgb, 255:red, 14; green, 75; blue, 147 }  ,draw opacity=1 ]   (419.27,163.67) ;
\draw [shift={(419.27,163.67)}, rotate = 0] [color={rgb, 255:red, 14; green, 75; blue, 147 }  ,draw opacity=1 ][fill={rgb, 255:red, 14; green, 75; blue, 147 }  ,fill opacity=1 ][line width=0.75]      (0, 0) circle [x radius= 3.35, y radius= 3.35]   ;
\draw [color={rgb, 255:red, 14; green, 75; blue, 147 }  ,draw opacity=1 ]   (215.09,163.33) ;
\draw [shift={(215.09,163.33)}, rotate = 0] [color={rgb, 255:red, 14; green, 75; blue, 147 }  ,draw opacity=1 ][fill={rgb, 255:red, 14; green, 75; blue, 147 }  ,fill opacity=1 ][line width=0.75]      (0, 0) circle [x radius= 3.35, y radius= 3.35]   ;
\draw [color={rgb, 255:red, 240; green, 107; blue, 0 }  ,draw opacity=1 ]   (419.57,361.67) ;
\draw [shift={(419.57,361.67)}, rotate = 0] [color={rgb, 255:red, 240; green, 107; blue, 0 }  ,draw opacity=1 ][fill={rgb, 255:red, 240; green, 107; blue, 0 }  ,fill opacity=1 ][line width=0.75]      (0, 0) circle [x radius= 3.35, y radius= 3.35]   ;
\draw [color={rgb, 255:red, 65; green, 117; blue, 5 }  ,draw opacity=1 ]   (215.39,361.33) -- (417.57,361.67) ;
\draw [shift={(419.57,361.67)}, rotate = 180.1] [color={rgb, 255:red, 65; green, 117; blue, 5 }  ,draw opacity=1 ][line width=0.75]    (10.93,-4.9) .. controls (6.95,-2.3) and (3.31,-0.67) .. (0,0) .. controls (3.31,0.67) and (6.95,2.3) .. (10.93,4.9)   ;
\draw [color={rgb, 255:red, 240; green, 107; blue, 0 }  ,draw opacity=1 ]   (215.39,361.33) ;
\draw [shift={(215.39,361.33)}, rotate = 0] [color={rgb, 255:red, 240; green, 107; blue, 0 }  ,draw opacity=1 ][fill={rgb, 255:red, 240; green, 107; blue, 0 }  ,fill opacity=1 ][line width=0.75]      (0, 0) circle [x radius= 3.35, y radius= 3.35]   ;
\draw [color={rgb, 255:red, 128; green, 128; blue, 128 }  ,draw opacity=1 ]   (213.07,163.34) -- (215.36,359.33) ;
\draw [shift={(215.39,361.33)}, rotate = 269.33] [color={rgb, 255:red, 128; green, 128; blue, 128 }  ,draw opacity=1 ][line width=0.75]    (10.93,-3.29) .. controls (6.95,-1.4) and (3.31,-0.3) .. (0,0) .. controls (3.31,0.3) and (6.95,1.4) .. (10.93,3.29)   ;
\draw [color={rgb, 255:red, 128; green, 128; blue, 128 }  ,draw opacity=1 ]   (419.27,163.67) -- (421.56,359.66) ;
\draw [shift={(421.58,361.66)}, rotate = 269.33] [color={rgb, 255:red, 128; green, 128; blue, 128 }  ,draw opacity=1 ][line width=0.75]    (10.93,-3.29) .. controls (6.95,-1.4) and (3.31,-0.3) .. (0,0) .. controls (3.31,0.3) and (6.95,1.4) .. (10.93,3.29)   ;
\draw [color={rgb, 255:red, 128; green, 128; blue, 128 }  ,draw opacity=1 ]   (215.09,163.33) .. controls (297.71,190.52) and (378.75,270.54) .. (420.94,358.21) ;
\draw [shift={(421.57,359.53)}, rotate = 244.53] [color={rgb, 255:red, 128; green, 128; blue, 128 }  ,draw opacity=1 ][line width=0.75]    (10.93,-3.29) .. controls (6.95,-1.4) and (3.31,-0.3) .. (0,0) .. controls (3.31,0.3) and (6.95,1.4) .. (10.93,3.29)   ;
\draw [color={rgb, 255:red, 128; green, 128; blue, 128 }  ,draw opacity=1 ]   (417.25,161.56) .. controls (323.14,204.8) and (263.07,261.82) .. (216.09,359.85) ;
\draw [shift={(215.39,361.33)}, rotate = 295.45] [color={rgb, 255:red, 128; green, 128; blue, 128 }  ,draw opacity=1 ][line width=0.75]    (10.93,-3.29) .. controls (6.95,-1.4) and (3.31,-0.3) .. (0,0) .. controls (3.31,0.3) and (6.95,1.4) .. (10.93,3.29)   ;

\draw (199.37,133.31) node [anchor=north west][inner sep=0.75pt]  [color={rgb, 255:red, 14; green, 75; blue, 147 }  ,opacity=1 ,rotate=-0.59]  {$G_{i( c)}$};
\draw (401.56,133.64) node [anchor=north west][inner sep=0.75pt]  [color={rgb, 255:red, 14; green, 75; blue, 147 }  ,opacity=1 ,rotate=-0.59]  {$G_{i( cd)}$};
\draw (203.47,369.95) node [anchor=north west][inner sep=0.75pt]  [color={rgb, 255:red, 65; green, 117; blue, 5 }  ,opacity=1 ,rotate=-0.59]  {$G_{\gamma }$};
\draw (308.78,138.51) node [anchor=north west][inner sep=0.75pt]  [color={rgb, 255:red, 208; green, 2; blue, 27 }  ,opacity=1 ,rotate=-0.59]  {$\psi _{d}$};
\draw (310.45,366.31) node [anchor=north west][inner sep=0.75pt]  [color={rgb, 255:red, 65; green, 117; blue, 5 }  ,opacity=1 ,rotate=-0.59]  {$\text{id}$};
\draw (259.15,223.68) node [anchor=north west][inner sep=0.75pt]  [color={rgb, 255:red, 14; green, 75; blue, 147 }  ,opacity=1 ,rotate=-0.59]  {$\psi _{c}$};
\draw (359.43,222.93) node [anchor=north west][inner sep=0.75pt]  [color={rgb, 255:red, 14; green, 75; blue, 147 }  ,opacity=1 ,rotate=-0.59]  {$\psi _{cd}$};
\draw (268.32,315.78) node [anchor=north west][inner sep=0.75pt]  [color={rgb, 255:red, 240; green, 107; blue, 0 }  ,opacity=1 ,rotate=-0.59]  {$\text{id}$};
\draw (194.12,254.81) node [anchor=north west][inner sep=0.75pt]  [color={rgb, 255:red, 128; green, 128; blue, 128 }  ,opacity=1 ]  {$\psi_{c}$};
\draw (356.32,316.78) node [anchor=north west][inner sep=0.75pt]  [color={rgb, 255:red, 240; green, 107; blue, 0 }  ,opacity=1 ,rotate=-0.59]  {$\text{id}$};
\draw (409.47,367.95) node [anchor=north west][inner sep=0.75pt]  [color={rgb, 255:red, 65; green, 117; blue, 5 }  ,opacity=1 ,rotate=-0.59]  {$G_{\gamma }$};
\draw (422.12,254.81) node [anchor=north west][inner sep=0.75pt]  [color={rgb, 255:red, 128; green, 128; blue, 128 }  ,opacity=1 ]  {$\psi _{cd}$};
\draw (307.47,271.95) node [anchor=north west][inner sep=0.75pt]  [color={rgb, 255:red, 0; green, 0; blue, 0 }  ,opacity=1 ,rotate=-0.59]  {$G_{\gamma }$};

\end{tikzpicture}

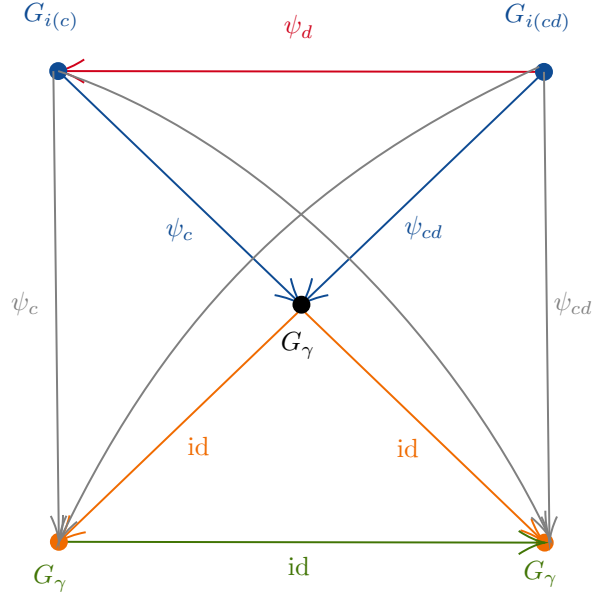
\captionof{figure}{The local complex of groups $L(\mathcal{Y}(\gamma))$.}
\label{fig:localCoG}
\end{center}

\begin{proposition}\label{prop_localisCoG}
    The \textit{local complex of group over $\gamma$} is a complex of groups.
\end{proposition}
\begin{proof}  
We first show that the data over pairs of composable morphisms align with Definition \ref{def_CoG} 3(a). Specifically, for all $(u,v)\in E^{(2)}(\mathcal{Y}(\gamma))$,
$$\text{Ad}(l_{u,v})\lambda_{uv}=\lambda_u \lambda_v.$$
The equations denoted by $(\Snowman[1.1])$ follow by Definition \ref{def_CoG} 3(a) for $G(\mathcal{Y})$.

    For all $((c_1,d_1),(c_2,d_2))\in E^{(2)}(\text{Lk}_\gamma)$,
    \begin{align*}
        \text{Ad}(l_{(c_1,d_1),(c_2,d_2)})\lambda_{(c_1, d_1d_2)}
        =\text{Ad}(g_{d_1,d_2})\psi_{d_1d_2}
        \overset{(\Snowman[1.1])}{=}\psi_{d_1}\psi_{d_2}
        =\lambda_{(c_1,d_1)} \lambda_{(c_2,d_2)}.
    \end{align*}
    
    For all $(\gamma*c,(c,d))\in (\{\gamma\}\times V(\text{Lk}_\gamma))\times E(\text{Lk}_\gamma))$,
    \begin{align*}
    \text{Ad}(l_{\gamma*c,(c,d)})\lambda_{\gamma*cd}
    =\text{Ad}(g_{c,d})\psi_{cd}\overset{(\Snowman[1.1])}{=}\psi_c \psi_d
    =\lambda_{\gamma*c} \lambda_{(c,d)}.
    \end{align*}

    For all $(b*c,(c,d))\in (V(\text{Lk}^\gamma)\times V(\text{Lk}_\gamma))\times E(\text{Lk}_\gamma)$,
    \begin{align*}
    \text{Ad}(l_{b*c,(c,d)})\lambda_{b*cd}
    =\text{Ad}(g_{c,d})\psi_{cd}
    \overset{(\Snowman[1.1])}{=}\psi_c \psi_d
    =
    \lambda_{b*c} \lambda_{(c,d)}.
    \end{align*}

    For all $((a,b),b*c)\in E(\text{Lk}^\gamma)\times(V(\text{Lk}^\gamma)\times V(\text{Lk}_\gamma))$,
    \begin{align*}
    \text{Ad}(l_{(a,b),b*c})\lambda_{ab*c}
    =\text{Ad}(e)\psi_c=\psi_c=\lambda_{(a,b)} \lambda_{b*c}.
    \end{align*}

    Now let $(b*\gamma,\gamma*c)\in(V(\text{Lk}^\gamma)\times \{\gamma\})\times ( \{\gamma\}\times V(\text{Lk}_\gamma))$.
    \begin{align*}
    \text{Ad}(l_{b*\gamma,\gamma*c})\lambda_{b*c}=\text{Ad}(e)\psi_{c}=\psi_c=\lambda_{b*\gamma} \lambda_{\gamma*c}.
    \end{align*}

    For all $((a,b),b*\gamma)\in E(\text{Lk}^\gamma)\times(V(\text{Lk}^\gamma)\times \{\gamma\})$,
    \begin{align*}
    \text{Ad}(l_{(a,b),b*\gamma})\lambda_{ab*\gamma}=\text{Ad}(e)\text{id}_{G_\gamma}=\lambda_{(a,b)} \lambda_{b*\gamma}.
    \end{align*}

For all $((a_1,b_1),(a_2,b_2))\in E^{(2)}(\text{Lk}^\gamma)$,
\begin{gather*}
    \text{Ad}(l_{((a_1,b_1),(a_2,b_2))}) \lambda_{(a_1a_2,b_2)} 
    =\text{Ad}(e)\text{id}_{G_\gamma}=
    \lambda_{(a_1,b_1)} \lambda_{(a_2,b_2)}.
\end{gather*}

Next we show that Condition \ref{def_CoG} 3(b) holds for all $ (u,v,w)\in E^{(3)}(\mathcal{Y}(\gamma))$, i.e.
$$\lambda_u(l_{v,w})l_{u,vw}=l_{u,v}l_{uv,w}.$$
The equations denoted by $(\Coffeecup)$ follow by Definition \ref{def_CoG} 3(b) for $G(\mathcal{Y})$.

For $((c_1,d_1), (c_2,d_2), (c_3,d_3))\in E^{(3)}(\text{Lk}_\gamma))$,
    \begin{align*}
        \lambda_{(c_1,d_1)}(l_{((c_2,d_2),(c_3,d_3))})l_{(c_1,d_1), (c_2,d_2d_3)}&=\psi_{d_1}(g_{d_2,d_3})g_{d_1,d_2d_3}\\&\overset{(\Coffeecup)}{=} g_{d_1,d_2}g_{d_1d_2,d_3}\\&=l_{(c_1,d_1), (c_2,d_2)}l_{(c_1,d_1d_2),(c_3,d_3)}.
    \end{align*}
    
For $(\gamma*c_1, (c_1,d_1), (c_2,d_2))\in (\{\gamma\}\times V(\text{Lk}_\gamma))\times E^{(2)}(\text{Lk}_\gamma)$ such that $c_1d_1=c_2$,
    \begin{align*}
        \lambda_{\gamma*c_1}(l_{(c_1,d_1), (c_2,d_2)})l_{\gamma*c_1, (c_1,d_1d_2)}&=\psi_{c_1}(g_{d_1,d_2})g_{c_1, d_1d_2}\\&\overset{(\Coffeecup)}{=} g_{c_1,d_1}g_{c_2,d_2}\\&=l_{\gamma*c_1, (c_1,d_1)}l_{\gamma*c_2,(c_2,d_2)}.
    \end{align*}

For $(b*c_1, (c_1,d_1), (c_2,d_2))\in (V(Lk^\gamma)\times V(\text{Lk}_\gamma))\times E^{(2)}(\text{Lk}_\gamma)$ such that $c_1d_1=c_2$,
    \begin{align*}
        \lambda_{b*c_1}(l_{(c_1,d_1), (c_2,d_2)})l_{b*c_1, (c_1,d_1d_2)}&=\psi_{c_1}(g_{d_1,d_2})g_{c_1, d_1d_2}\\&\overset{(\Coffeecup)}{=} g_{c_1,d_1}g_{c_2,d_2}\\&=l_{b*c_1, (c_1,d_1)}l_{b*c_2,(c_2,d_2)}.
    \end{align*}

For $(b*\gamma, \gamma*c, (c,d))\in (V(\text{Lk}^\gamma)\times \{\gamma\})\times (\{\gamma\}\times V(\text{Lk}_\gamma))\times E(\text{Lk}_\gamma)$,
    \begin{gather*}
        \lambda_{b*\gamma}(l_{\gamma*c,(c,d)})l_{b*\gamma,\gamma*cd}= \text{id}_{G_\gamma}(g_{c,d})e=g_{c,d}=l_{b*\gamma,\gamma*c}l_{b*c,(c,d)}.
    \end{gather*}

For $((a,b), b*c, (c,d))\in E(\text{Lk}^\gamma)\times (V(\text{Lk}^\gamma)\times V(\text{Lk}_\gamma))\times E(\text{Lk}^\gamma)$,
    \begin{gather*}
        \lambda_{(a,b)}(l_{b*c,(c,d)})l_{(a,b),b*cd}=\text{id}_{G_\gamma}(g_{c,d})e=g_{c,d}=l_{(a,b),b*c}l_{ab*c,(c,d)}.
    \end{gather*}

For $((a,b), b*\gamma, \gamma*c)\in E(\text{Lk}^\gamma)\times (V(\text{Lk}^\gamma)\times\{\gamma\})\times(\{\gamma\}\times V(\text{Lk}_\gamma))$,
    \begin{gather*}
        \lambda_{(a,b)}(l_{b*\gamma,\gamma*c})l_{(a,b),b*c}=e=l_{(a,b), b*\gamma}l_{ab*\gamma,\gamma*c}.
    \end{gather*}

For $((a_1,b_1), (a_2,b_2), b_2*\gamma)\in E^{(2)}(\text{Lk}^\gamma)\times (V(\text{Lk}^\gamma)\times\{\gamma\})$,
    \begin{gather*}
        \lambda_{(a_1,b_1)}(l_{(a_2,b_2), b_2*\gamma})l_{(a_1,b_1),a_2b_2*\gamma}=e=l_{(a_1,b_1), (a_2,b_2)}l_{(a_1,b_1b_2),b_2*\gamma}.
    \end{gather*}

For $((a_1,b_1), (a_2,b_2), b_2*c)\in E^{(2)}(\text{Lk}^\gamma)\times (V(\text{Lk}^\gamma)\times V(\text{Lk}_\gamma))$,
    \begin{gather*}
        \lambda_{(a_1,b_1)}(l_{(a_2,b_2), b_2*c})l_{(a_1,b_1),a_2b_2*c}=e=l_{(a_1,b_1), (a_2,b_2)}l_{(a_1,b_1b_2),b_2*c}.
    \end{gather*}

For $((a_1,b_1), (a_2,b_2), (a_3,b_3))\in E^{(3)}(\text{Lk}^\gamma)$,
    \begin{gather*}
        \lambda_{(a_1,b_1)}(l_{(a_2,b_2), (a_3,b_3)})l_{(a_1,b_1),(a_2a_3,b_3)}=e=l_{(a_1,b_1), (a_2,b_2)}l_{(a_1,b_1b_2),(a_3,b_3)}.
    \end{gather*}
\end{proof}

\subsection{The local development revisited}
In this section we prove that a local complex of groups reflects the local data of a complex of groups. This is the content of Theorem \ref{thm_localCoG}, which we recall below:
\thmlocalCoG*

\begin{proposition}
\label{prop_pi1localCoG}
The local homomorphism
$(\iota_T)_\gamma:L_\gamma\longrightarrow \pi_1(L(\mathcal{Y}(\gamma)), T)$ is an isomorphism.
\end{proposition}
\begin{proof}
Let $T$ be the maximal tree in $|\mathcal{Y}(\gamma)|^{(1)}$ containing all morphisms of the form $$b*\gamma\in V(\text{Lk}^\gamma)\times\{\gamma\}\text{ and }\gamma*c\in\{\gamma\}\times V(\text{Lk}_\gamma).$$ Observe that $T=(V(\text{Lk}^\gamma)\times\{\gamma\})\cup (\{\gamma\}\times V(\text{Lk}_\gamma))$. 
By the presentation of the fundamental group in Definition \ref{def_pi1_presentation}, there exists a surjection from the free product of its generators to the group itself
$$\Pi: *_{\mu\in V(\mathcal{Y}(\gamma))} L_\mu * E^{\pm}(\mathcal{Y}(\gamma))\longrightarrow \pi_1(L(\mathcal{Y}(\gamma)), T).$$
We first show surjectivity of $(\iota_T)_\gamma$ by showing that $\Pi$ factors through $(\iota_T)_\gamma:L_\gamma\to\pi_1(L(\mathcal{Y}(\gamma)), T).$

Recall that $\pi_1(L(\mathcal{Y}(\gamma)), T)$ is generated by
$$\left\{\bigsqcup_{\mu\in V(\mathcal{Y}(\gamma))} L_\mu\bigsqcup E^{\pm}(\mathcal{Y}(\gamma))\right\}$$
over the set of relations
\begin{center} 
$\left\{\begin{aligned}
    &(u^+)^{-1}=u^- &&u\in E(\mathcal{Y}(\gamma))\\
    &(\gamma*c)^\pm, (b*\gamma)^\pm=e &&\gamma*c\in\{\gamma\}\times V(\text{Lk}_\gamma), b*\lambda\in V(\text{Lk}^\gamma)\times\{\gamma\}\\
    &\lambda_{\gamma*c}(g)=(\gamma*c)^+g(\lambda*c)^- &&c\in V(\text{Lk}_\gamma)\\
    &u^+v^+= l_{u,v}(uv)^+ &&(u,v)\in E^{(2)}(\mathcal{Y}(\gamma))
\end{aligned}
\right\}.$
\end{center}
%

Since $(\gamma*c)^\pm\mapsto e$ for all $\gamma*c\in\{\gamma\}\times V(\text{Lk}_\gamma)$,
\begin{gather}\label{eq_iotaTgamma1}
    L_c=\text{Ad}((\gamma*c)^+)(L_c)=\lambda_{\gamma*c}(L_c)\subset L_\gamma.
\end{gather}
In addition, by definition of $\lambda_{b*\gamma}$ for all $b*\gamma\in V(\text{Lk}^\gamma)\times \{\gamma\}$, 
\begin{gather}\label{eq_iotaTgamma2}
    L_\gamma= \text{Id}_{L_\gamma}(L_\gamma)=\lambda_{b*\gamma} \subset L_b.
\end{gather}
Hence, all local groups are identified with subgroups of $L_\gamma$.

It remains to show injectivity of $E^\pm(\mathcal{Y}(\gamma))\to L_\gamma$ under this projection. By the second relation, $(\lambda*c)^\pm, (b*\lambda)^\pm\mapsto e$. The remaining generators come from morphisms of the form $(c,d)\in E(\text{Lk}_\gamma), b*c\in V(\text{Lk}^\gamma)\times V(\text{Lk}_\gamma)$ and $(a,b)\in E(\text{Lk}^\gamma)$. Applying the last relation to the pairs of morphisms $(\gamma*c, (c,d))\in (\{\gamma\}\times V(\text{Lk}_\gamma))\times E(\text{Lk}_\gamma), (b*\gamma, \gamma*c)\in (V(\text{Lk}^\gamma)\times \{\gamma\})\times ( \{\gamma\}\times V(\text{Lk}_\gamma))$ and $((a,b),b*\gamma)\in E(\text{Lk}^\gamma)\times (V(\text{Lk}^\gamma)\times \{\gamma\})$, we see that
\begin{gather}
    (c,d)^+=(\gamma*c)^+(c,d)^+=l_{\gamma*c, (c,d)}(\gamma*cd)^+=l_{\gamma*c, (c,d)}\in L_\gamma, \label{eq_iotaTgamma3}
\\(b*c)^+=(b*\gamma)^+(\gamma*c)^+(l_{b*\gamma, \gamma*c})^{-1}=(l_{b*\gamma, \gamma*c})^{-1}\in L_b\subset L_\gamma,\label{eq_iotaTgamma4} \text{ and}
   \\
    (a,b)^+=(a,b)^+(b*\gamma)^+=l_{(a,b), b*\gamma}(ab*\gamma)^+
    =l_{(a,b), b*\gamma}\in L_{ab}\subset L_\gamma.\label{eq_iotaTgamma5}
\end{gather}

Consequently, $\Pi$ factors through
$(\iota_T)_\gamma$, making it a surjection.

Next, we show that $(\iota_T)_\gamma$ is a monomorphism. We do so by defining a morphism
$\Theta:L(\mathcal{Y}(\gamma))\to L_\gamma$ whose induced homorphism on fundamental groups is ${(\iota_T)_\gamma}^{-1}$, i.e., it factors through the identity map on $L_\gamma$.
We define $\Theta$ as follows:
\begin{align*}
    \Theta_-&=\begin{cases}
            \lambda_{\gamma*c}:L_c\to L_\gamma &c\in V(\text{Lk}_\gamma)\\
            \text{Id}_{L_\gamma}:L_\gamma\to L_\gamma &\gamma\in \{\gamma\}\\
            \text{Id}_{L_b}:L_b\to L_b &b\in V(\text{Lk}^\gamma)\\
        \end{cases}
\\
\Theta(-)&=\begin{cases}
            l_{\gamma*c, (c,d)} &(c,d)\in E(\text{Lk}_\gamma)\\
            e &\gamma*c\in \{\gamma\}\times V(\text{Lk}_\gamma)\\
            {l_{b*\gamma, \gamma*c}}^{-1} &b*c\in V(\text{Lk}^\gamma)\times V(\text{Lk}_\gamma)\\
            e &b*\gamma\in V(\text{Lk}^\gamma)\times\{\gamma\}\\
            l_{(a,b), b*\gamma} &(a,b)\in E(\text{Lk}^\gamma)
        \end{cases}
\end{align*}
\begin{claim}
    The morphism $\Theta:L(\mathcal{Y}(\gamma))\to L_\gamma$ is a morphism of complexes of groups. 
\end{claim}
\begin{proof}[Proof of claim.]
    We first show that for all $u\in E(\mathcal{Y}(\gamma))$, $\text{Ad}(\Theta(u))\Theta_{i(u)}=\Theta_{t(u)}\lambda_{u}$.
    
    Equations annotated with $(\Candle)$ follow from Definition \ref{def_CoG} 3(a) for $L(\mathcal{Y}(\gamma))$. 
    \begin{align*}
    \text{Ad}(\Theta(c,d))\Theta_{i((c,d))}
    &=\text{Ad}(l_{\gamma*c, (c,d)})\lambda_{\gamma*cd}\\
    &\overset{(\Candle)}{=}\lambda_{\gamma*c}\lambda_{(c,d)}\\
    &=\Theta_{t((c,d))}\lambda_{(c,d)}.
    \end{align*}
    \begin{align*}
    \text{Ad}(\Theta(\gamma*c))\Theta_{i(\gamma*c)}
    &=\text{Ad}(e)\lambda_{\gamma*c}\\
    &=\text{id}_{L_\gamma}\lambda_{\gamma*c}\\
    &=\Theta_{t(\gamma*c)}\lambda_{\gamma*c}.
    \end{align*}
    \begin{align*}
    \text{Ad}(\Theta(b*c))\Theta_{i(b*c)}
    &=\text{Ad}({l_{b*\gamma, \gamma*c}}^{-1})\lambda_{\gamma*c}\\
    &=\text{Ad}({l_{b*\gamma, \gamma*c}}^{-1})\text{id}_{L_\gamma}\lambda_{\gamma*c}\\
    &=\text{Ad}({l_{b*\gamma, \gamma*c}}^{-1})\lambda_{b*\gamma}\lambda_{\gamma*c}\\
    &\overset{(\Candle)}{=}\text{id}_{L_b}\lambda_{b*c}\\
    &=\Theta_{t(b*c)}\lambda_{b*c}.
    \end{align*}
    \begin{align*}
    \text{Ad}(\Theta(b*\gamma))\Theta_{i(b*\gamma)}
    &=\text{Ad}(e)\text{id}_{L_\gamma}\\
    &=\text{id}_{L_b}\text{id}_{L_\gamma}\\
    &=\Theta_{t(b*\gamma)}\lambda_{b*\gamma}.
    \end{align*}
    \begin{align*}
    \text{Ad}(\Theta((a,b)))\Theta_{i(a,b)}
    &=\text{Ad}(l_{(a,b), b*\gamma})\text{id}_{L_b}\\
    &=\text{id}_{L_{ab}}\text{id}_{L_b}\\
    &=\Theta_{t((a,b))}\lambda_{(a,b)}.
    \end{align*}

    Next we show that for all $(u,v)\in E^{(2)}(\mathcal{Y}(\gamma))$, $\Theta_{t(u)}(l_{u,v})\Theta(uv)=\Theta(u)\Theta(v)$.
    Equations annotated with $(\Strichmaxerl)$ follow from Definition \ref{def_CoG} 3(b) for $L(\mathcal{Y}(\gamma))$. 

    \begin{align*}
    \Theta_{t((c_1,d_1))}(l_{(c_1,d_1),(c_2,d_2)})\Theta((c_1,d_1d_2))
    &=\lambda_{\gamma*c_1}(l_{((c_1,d_1),(c_2,d_2))})l_{\gamma*c_1, (c_1,d_1d_2)}\\
    &\overset{(\Strichmaxerl)}{=}l_{\gamma*c_1, (c_1,d_1)}l_{\gamma*c_2, (c_2,d_2)}\\
    &=\Theta((c_1,d_1))\Theta((c_2,d_2)).
    \end{align*}
    \begin{align*}
    \Theta_{t(\gamma*c)}(l_{\gamma*c, (c,d)})\Theta(\gamma*cd)
    &=\text{id}_{L_\gamma}(l_{\gamma*c, (c,d)})e\\
    &=el_{\gamma*c, (c,d)}\\
    &=\Theta(\gamma*c)\Theta((c,d)).
    \end{align*}
    \begin{align*}
    \Theta_{t(b*c)}(l_{b*c,(c,d)})\Theta(b*cd)
    &=\text{id}_{L_b}(l_{b*c, (c,d)}){l_{b*\gamma, \gamma*cd}}^{-1}\\
    &=\lambda_{b*\gamma}({l_{b*\gamma, \gamma*c}}^{-1})l_{\gamma*c, (c,d)}\\
    &\overset{(\Strichmaxerl)}{=}{l_{b*\gamma, \gamma*c}}^{-1}l_{\gamma*c, (c,d)}\\
    &=\Theta(b*c)\Theta((c,d)).
    \end{align*}
    \begin{align*}
    \Theta_{t((a,b))}(l_{(a,b),b*c})\Theta(ab*c)
    &=\text{id}_{L_{ab}}(l_{(a,b),b*c)}{l_{ab*\gamma, \gamma*c}}^{-1}\\
    &=e\\
    &=l_{(a,b),b*c}{l_{b*\gamma,\gamma*c}}^{-1}\\
    &=\Theta((a,b))\Theta(b*c).
    \end{align*}
    \begin{align*}
    \Theta_{t(b*\gamma)}(l_{b*\gamma, \gamma*c})\Theta(b*c)
    &=\text{id}_{L_b}(l_{b*\gamma, \gamma*c}){l_{b*\gamma, \gamma*c}}^{-1}\\
    &=e\\
    &=\Theta(b*\gamma)\Theta(\gamma*c).
    \end{align*}
    \begin{align*}
    \Theta_{t((a,b))}(l_{(a,b),b*\gamma})\Theta(ab*\gamma)
    &=\text{id}_{L_{ab}}(l_{(a,b),b*\gamma}){l_{b*\gamma, \gamma*c}}^{-1}\\
    &=l_{(a,b),b*\gamma}e\\
    &=\Theta((a,b))\Theta(b*\gamma).
    \end{align*}
    \begin{align*}
    \Theta_{t((a_1,b_1))}(l_{(a_1,b_1),(a_2,b_2)})\Theta(a_1a_2,b_2)
    &=\text{id}_{L_{a_1b_1}}(l_{(a_1,b_1),b_1*\gamma})l_{(a_1a_2,b_2),b_2*\gamma}\\
    &=l_{(a_1,b_1),b_1*\gamma}l_{(a_2,b_2),b_2*\gamma}\\
    &=\Theta((a_1,b_1))\Theta((a_2,b_2)).
    \end{align*}

We have thus proven the claim that $\Theta$ is a morphism from a complex of groups to a group.
\end{proof}
By Proposition \ref{prop_pi1inducedtoG}, $\Theta_*: \pi_1(L(\mathcal{Y}(\gamma)),T)\to L_\gamma$ sends the generators $g\in L_\mu$ to $\Theta_\mu(g)$ and $u^+$ to $\Theta(u)$. Also
recall Equations (\ref{eq_iotaTgamma1}) to (\ref{eq_iotaTgamma5}), which describe the homomorphism $(\iota_T)_\gamma$ by relating subgroups and twisting elements in $L_\gamma$ to the generators of $\pi_1(L(\mathcal{Y}(\gamma)),T)$. From this data, we see that $\Theta_*={(\iota_T)_\gamma}^{-1}$. In particular, the following diagram commutes:
\begin{center}
    \begin{tikzcd}[row sep=huge,column sep=huge]
    L_\gamma \arrow[r, two heads, "(\iota_T)_\gamma"] \arrow[dr, "\text{id}"]
    & \pi_1(L(\mathcal{Y}(\gamma)), T) \arrow[d, "\Theta_*"] \\
     & L_\gamma
  \end{tikzcd}
\end{center}
  If there exists a nontrivial element $g\in \text{ker}((\iota_T)_\gamma)$, then $g\in \text{ker}(\Theta_*\circ (\iota_T)_\gamma)=\text{ker}(\text{id})$.
\end{proof}

\begin{proposition}\label{prop_localdevelopability}
    A local complex of groups is always developable.
\end{proposition}
\begin{proof}
    Recall by Definition \ref{def_morphismtoG} that 
    \begin{align*}
        \Theta_c&=\Theta_\gamma \lambda_{\gamma*c} &c\in V(\text{Lk}_\gamma)\\
        \Theta_b&=\Theta_\gamma{\lambda_{b*\gamma}}^{-1}=\Theta_\gamma &b\in V(\text{Lk}^\gamma)
    \end{align*}
    
    By our construction of $\Theta$, $\Theta_\gamma$ and $\lambda_{\gamma*c}$ are monomorphisms. Since all local homomorphisms of $\Theta$ are injective, by the developability criterion in Theorem \ref{thm_developability_criterion}, the local complex of groups $L(\mathcal{Y}(\gamma))$ is developable.
\end{proof}

\begin{proposition}\label{prop_development_localCoG}
The \textit{local development of $\gamma\in V(\mathcal{Y})$} is isomorphic to the development of $\mathcal{Y}(\gamma)$ with respect to the natural morphism to its fundamental group, i.e.,
\begin{equation*}
D(\mathcal{Y}(\gamma), \iota_T)\cong {\mathcal{Y}(\tilde{\gamma})}.
\end{equation*}
\end{proposition}
\begin{proof}
    We first show that \begin{center}
        $V(D(\mathcal{Y}(\gamma),\iota_T))=V(\mathcal{Y}(\tilde{\gamma}))$ and $E(D(\mathcal{Y}(\gamma),\iota_T))=E(\mathcal{Y}(\tilde{\gamma})).$
    \end{center}
    In both of these cases, the first equality follows from the explicit construction of the development in \cite[Theorem III.$\mathcal{C}$.2.13]{bh}. 
    Since $(\iota_T)_\mu$ is injective for all $\mu\in V(\mathcal{Y}(\gamma))$, we express the coset
    \begin{gather*}
        h(\iota_T)_\mu(L_\mu)
    \in \rightQ{L_\gamma}{(\iota_T)_\mu(L_\mu)}\text{ as } hL_\mu\in \rightQ{L_\gamma}{L_\mu}
    \end{gather*}
    and 
    \begin{gather*}g\psi_c(G_{i(c)})\in\rightQ{G_\gamma}{\psi_c(G_{i(c)})}\text{ as }
        gG_{i(c)}\in \rightQ{G_\gamma}{G_{i(c)}}.
    \end{gather*}
    Observe
    \begin{align*}
    &V(D(\mathcal{Y}(\gamma),\iota_T))
    \\=&\{(hL_\mu,\mu)|\mu\in V(\mathcal{Y}(\gamma)),hL_\mu\in \rightQ{L_\gamma}{L_\mu}\}
    \\=&\{(hL_c,c)| c\in V(\text{Lk}_\gamma), hL_c\in \rightQ{L_\gamma}{L_c}\}
    \\&\sqcup
    \{(hL_\gamma, \gamma)| hL_\gamma\in \rightQ{L_\gamma}{L_\gamma}\}
    \\&\sqcup 
    \{(hL_b,b)| b\in V(\text{Lk}^\gamma), hL_b\in \rightQ{L_\gamma}{L_b}\}
    \\=&\{(gG_{i(c)},c)| t(c)=\gamma, gG_{i(c)}\in \rightQ{G_\gamma}{G_{i(c)}}\}
    \\&\sqcup
    \{\gamma\}
    \\&\sqcup 
    \{b\in E(\mathcal{Y})| i(b)=\gamma\}
    \\=&V(\text{Lk}_{\tilde{\gamma}})\sqcup \{\gamma\}\sqcup V(\text{Lk}^\gamma)
    \\=& V(\mathcal{Y}(\tilde{\gamma})),
    \end{align*}
    where $V(\text{Lk}_{\tilde{\gamma}})$ and $V(\text{Lk}^\gamma)$ are described in Definition \ref{def_Y(gamma)} and Definition \ref{def_localdevelopment} respectively. Also,\begin{align*}
&E(D(\mathcal{Y}(\gamma),\iota_T))\\=&\{(hL_{i(u)}, u)| u\in E(\mathcal{Y}(\gamma)), h L_{i(u)}\in \rightQ{L_\gamma}{L_{i(u)}}\}\\
        =&\{(hL_{cd},c,d)|(c,d)\in E(\text{Lk}_\gamma),hL_{cd})\in \rightQ{L_\gamma}{L_{cd}}\}\\
        &\sqcup \{(hL_c, \gamma*c)| \gamma*c\in \{\gamma\}\times V(\text{Lk}_\gamma)\}\\
        &\sqcup \{(hL_c, b*c)|b*c\in V(\text{Lk}^\gamma)\times V(\text{Lk}_\gamma), hL_c\in \rightQ{L_\gamma}{L_c}\}\\
        &\sqcup \{(hL_b, b*\gamma)| b*\gamma\in V(\text{Lk}^\gamma)\times \{\gamma\}, hL_\gamma\in \rightQ{L_\gamma}{L_\gamma}\}\\
        &\sqcup \{(hL_{ab},a,b)|(a,b)\in E(\text{Lk}^\gamma),hL_{ab}\in \rightQ{L_\gamma}{L_\gamma}\}\\
        =&
        \{(gG_{i(d)},c,d)|(c,d)\in E^{(2)}(\mathcal{Y}), t(c)=\gamma,gG_{i(d)}\in \rightQ{G_\gamma}{G_{i(d)}}\}\\
        &\sqcup \{\gamma*(gG_{i(c)}, c)| c\in V(\text{Lk}_\gamma),
        gG_{i(c)}\in \rightQ{G_\gamma}{G_{i(c)}}
        \}\\
        &\sqcup \{b*(gG_{i(c)}, c)|b\in  V(\text{Lk}^\gamma), (gG_{i(c)}, c)\in V(\text{Lk}_{\tilde{\gamma}}), gG_{i(c)}\in \rightQ{G_\gamma}{G_{i(c)}}\}\\
        &\sqcup \{b*\gamma\in E(\mathcal{Y})\times\{\gamma\}| i(b)=\gamma\}\\
        &\sqcup \{(a,b)\in E^{(2)}(\mathcal{Y})| i(b)=\gamma\}  
        \\=&E(\text{Lk}_{\tilde{\gamma}}) 
        \sqcup (\{\gamma\}\times V(\text{Lk}_{\tilde{\gamma}}))
        \sqcup (V(\text{Lk}^\gamma)\times V(\text{Lk}_{\tilde{\gamma}}))
        \sqcup (V(\text{Lk}^{\gamma})\times \{\gamma\})
        \sqcup E(\text{Lk}^{\gamma}) 
        \\=&E(\mathcal{Y}(\tilde{\gamma})).
    \end{align*}
    Note that the second last equation follows from the description of each type of morphism of $\mathcal{Y}(\tilde{\gamma})$ in Definition \ref{def_localdevelopment}.

    Each type of morphism in $D(\mathcal{Y}(\gamma),\iota_T)$ has the following source and target maps as described in the proof of \cite[Theorem 2.13]{bh}\footnote{We also use that $\iota_T(u)=u^+=l_{-,-}$, where $l_{-,-}$ is the twisting element in $L_\gamma$ in Equations (\ref{eq_iotaTgamma3}) - (\ref{eq_iotaTgamma5}).}:
    \begin{align*}
    i:E(D(\mathcal{Y}(\gamma),\iota_T))&\longrightarrow V(D(\mathcal{Y}(\gamma),\iota_T))\\
    (hL_{cd},c,d)&\longmapsto (hL_{cd},cd) \\
    (hL_c,\gamma*c)&\longmapsto (hL_c,c) \\
    (hL_c,b*c)&\longmapsto (hL_c,c) \\
    (hL_\gamma,b*\gamma)&\longmapsto (hL_\gamma,\gamma) \\
    (hL_b,a,b)&\longmapsto (hL_b,b)
\\
    t:E(D(\mathcal{Y}(\gamma),\iota_T))&\longrightarrow V(D(\mathcal{Y}(\gamma),\iota_T))\\
    (hL_{cd},c,d)&\longmapsto (h{l_{\gamma*c,(c,d)}}^{-1}L_{c},c)
    \\
    (hL_c,\gamma*c)&\longmapsto (hL_\gamma,\gamma)
    \\
    (hL_c,b*c)&\longmapsto 
    (hL_b,b)
    \\
    (hL_\gamma,b*\gamma)&\longmapsto (hL_b,b)
    \\
        (hL_b,a,b)&\longmapsto (h{l_{(a,b), b*\gamma}}^{-1}L_{ab},ab)
    \end{align*}
Comparing this with $i,t:E(\mathcal{Y}(\tilde{\gamma}))\to V(\mathcal{Y}(\tilde{\gamma}))$ in Definition \ref{def_localdevelopment}, we see that source and target maps of both scwols are equivalent. Composition follows by the equivalence of morphisms and the equivalence of their source and target maps.

Hence, we conclude that $D(\mathcal{Y}(\gamma), \iota_T)$ is isomorphic to $\mathcal{Y}(\tilde{\gamma})$.
\end{proof}

\begin{proof}[Proof of Theorem \ref{thm_localCoG}]
    The proof follows immediately from Proposition \ref{prop_pi1localCoG}, Proposition \ref{prop_localdevelopability} and Proposition \ref{prop_development_localCoG}.
\end{proof}

\subsection{A functorial version of the developability criterion}

\begin{construction}\label{construction_morphismfunctorial}
    There exists a morphism $\Sigma:L(\mathcal{Y}(\gamma))\to G(\mathcal{Y})$ of complexes of groups over the morphism $h:\mathcal{Y}(\gamma)\to\mathcal{Y}$\footnote{This morphism is defined in Proposition \ref{prop_morphism_Y(gamma)toY}.}   of scwols, defined by
\begin{align*}
    \Sigma_-=
    \begin{cases}
        \text{id}_{L_c} &c\in V(\text{Lk}_\gamma)\\
        \text{id}_{L_\gamma} &\gamma\in\{\gamma\}\\
       \psi_b &b\in V(\text{Lk}^\gamma)
\end{cases}
&&
\Sigma(-)=
    \begin{cases}
     e &(c,d)\in E(\text{Lk}_\gamma)\\
     e &\gamma*c\in \{\gamma\}\times V(\text{Lk}_\gamma)\\
     g_{b,c} &b*c\in V(\text{Lk}^\gamma)\times V(\text{Lk}_\gamma)\\
     e &b*\gamma\in V(\text{Lk}^\gamma)\times \{\gamma\}\\
     {g_{a,b}}^{-1} &(a,b)\in E(\text{Lk}^\gamma)
\end{cases}
\end{align*}
\end{construction}

\begin{proof}
We first show that the data over each morphism is compatible with $\Sigma$ in the sense of Definition \ref{morphism_CoG} (2.1):
    $$\text{Ad}(\Sigma(u))\psi_{h(u)}\Sigma_{i(u)}=\Sigma_{t(u)}\lambda_u,\forall u\in E(\mathcal{Y}(\gamma)).$$
    To prove this condition, we make reference to the data of the morphism $\Sigma$ above, the morphism $h:\mathcal{Y}(\gamma)\to\mathcal{Y}$ in Proposition \ref{prop_morphism_Y(gamma)toY}, the source and target morphisms $i,t:E(\mathcal{Y}(\gamma))\to V(\mathcal{Y}(\gamma))$ in Definition \ref{def_Y(gamma)} and the monomorphisms $\lambda_u$ for each $u\in E(\mathcal{Y}(\gamma))$ from the data of the complex of groups $L(\mathcal{Y}(\gamma))$ in Definition \ref{def_localCoG}. Equations denoted by $(\fryingpan)$ hold by Definition \ref{def_CoG} 3(a).
    
    For all $(c,d)\in E(\text{Lk}_\gamma)$,
    $$\text{Ad}(\Sigma(c,d))\psi_{h((c,d))}\Sigma_{i((c,d))}=\text{Ad}(e)\psi_d\text{id}_{G_{cd}}=\psi_d=\Sigma_{t(c,d)}\lambda_{(c,d)}$$

    For all $\gamma*c\in \{\gamma\}\times V(\text{Lk}_\gamma)$,
    \begin{gather*}
    \text{Ad}(\Sigma(\gamma*c))\psi_{h(\gamma*c)}\Sigma_{i(\gamma*c)}=\text{Ad}(e)\psi_c\text{id}_{G_c}=\psi_c=\Sigma_{t(\gamma*c)}\lambda_{\gamma*c}.
    \end{gather*}

    For all $b*c\in V(\text{Lk}^\gamma)\times V(\text{Lk}_\gamma)$,
    \begin{align*}
        \text{Ad}(\Sigma(b*c))\psi_{h(b*c)}\Sigma_{i(b*c)}&=\text{Ad}(g_{b,c})\psi_{bc}\text{id}_{G_c}
        \\
        &=\text{Ad}(g_{b,c})\psi_{bc}\\
        &\overset{(\fryingpan)}{=}\psi_b\psi_c\\
        &=\Sigma_{t(b*c)}\lambda_{b*c}.
    \end{align*}

    For all $b*\gamma\in V(\text{Lk}^\gamma)\times \{\gamma\}$,
    $$\text{Ad}(\Sigma(b*\gamma))\psi_{h(b*\gamma)}\Sigma_{i(b*\gamma)}=\text{Ad}(e)\psi_b\text{id}_{G_\gamma}=\psi_b=\Sigma_{t(b*\gamma)}\lambda_{b*\gamma}.$$

    For all $(a,b)\in E(\text{Lk}^\gamma)$,
    \begin{align*}\text{Ad}(\Sigma((a,b))\psi_{h((a,b))}\Sigma_{i((a,b))}&=\text{Ad}({g_{a,b}}^{-1})\psi_a\psi_b\\&\overset{(\fryingpan)}{=}\psi_{ab}\\&=\psi_{ab}\text{id}_{G_\gamma}\\&=\Sigma_{t((a,b))}\lambda_{(a,b)}.\end{align*}

    It remains to show that the data over pairs of morphisms in $\mathcal{Y}(\gamma)$ is compatible with $\Sigma$ as in Definition \ref{morphism_CoG} (2.2). Specifically, for all $(u,v)\in E^{(2)}(\mathcal{Y}(\gamma))$,
    $$\Sigma_{t(u)}(l_{u,v})\Sigma(uv)=\Sigma(u)\psi_{h(u)}(\Sigma(v))g_{h(u), h(v)}.$$
    Equations denoted by $(\pot)$ hold by the second condition for twisting elements in the definition of the complex of groups applied to $G(\mathcal{Y})$ (see Definition \ref{def_CoG} 3(b)).
    
    For all $((c_1,d_1),(c_2,d_2))\in E^{(2)}(\text{Lk}_\gamma)$ where $c_2=c_1d_1$,
    \begin{align*}
        &\Sigma_{t(c_1,d_1)}(l_{(c_1,d_1),(c_2,d_2)})\Sigma((c_1, d_1d_2))\\=&\text{id}_{G_{c_1}}(g_{d_1,d_2})e
        \\=&g_{d_1,d_2}\\=&\Sigma((c_1,d_1))\psi_{h((c_1,d_1))}(\Sigma(c_2,d_2))g_{h((c_1,d_1)),h((c_2,d_2))}.
    \end{align*}
    
    For all $(\gamma*c,(c,d))\in (\{\gamma\}\times V(\text{Lk}_\gamma)) \times E(\text{Lk}_\gamma)\subset E^{(2)}(\mathcal{Y}(\gamma))$,
    \begin{align*}\Sigma_{t(\gamma*c)}(l_{\gamma*c,(c,d)})\Sigma(\gamma*cd)=&\text{id}_{G_\gamma}(g_{c,d})e
    \\=&g_{c,d}
    \\=&
    \Sigma(\gamma*c)\psi_{h(\gamma*c)}(\Sigma((c,d)))g_{h(\gamma*c),h((c,d))}.
    \end{align*}

    For all $(b*c,(c,d))\in (V(\text{Lk}^\gamma)\times V(\text{Lk}_\gamma))\times E(\text{Lk}_\gamma)\subset E^{(2)}(\mathcal{Y}(\gamma))$,
    \begin{align*}
    \Sigma_{t(b*c)}
    (l_{b*c,(c,d)})
    \Sigma(b*cd)=& \psi_b
    (g_{c,d})
    g_{b,cd}
    \\\overset{(\pot)}{=}&g_{b,c}g_{bc,d}
    \\=&\Sigma(b*c)\psi_{h(b*c)}(\Sigma(c,d))g_{h(b*c),h((c,d))}.
    \end{align*}

    For all $(b*\gamma,\gamma*c)\in (V(\text{Lk}^\gamma)\times \{\gamma\})\times (\{\gamma\}\times V(\text{Lk}_\gamma))\subset E^{(2)}(\mathcal{Y}(\gamma))$,
    \begin{align*}
        \Sigma_{t(b*\gamma)}(l_{b*\gamma,\gamma*c})\Sigma(b*c)&=\psi_b(e)g_{b,c}
        \\&=g_{b,c}
        \\&=\Sigma(b*\gamma)\psi_{h(b*\gamma)}(\Sigma(\gamma*c))g_{h(b*\gamma),h(\gamma*c)}.
    \end{align*}

    For all $((a,b),b*c)\in  E(\text{Lk}^\gamma)\times (V(\text{Lk}^\gamma)\times V(\text{Lk}_\gamma)) \subset E^{(2)}(\mathcal{Y}(\gamma))$,
    \begin{align*}
        \Sigma_{t((a,b))}(l_{(a,b),b*c})\Sigma(ab*c)&=\psi_{ab}(e)g_{ab,c}
        \\&=g_{ab,c}
        \\&\overset{(\pot)}{=}{g_{a,b}}^{-1}\psi_a(g_{b,c})g_{a,bc}
        \\&=\Sigma((a,b))\psi_{h((a,b))}(\Sigma(b*c))g_{h((a,b)),h(b*c)}.
    \end{align*}

    For all $((a,b),b*\gamma)\in  E(\text{Lk}^\gamma)\times V(\text{Lk}^\gamma)\times\{\gamma\}\subset E^{(2)}(\mathcal{Y}(\gamma))$,
    \begin{align*}
        \Sigma_{t((a,b))}(l_{(a,b),b*\gamma})\Sigma(ab*\gamma)
        &=\psi_{ab}(e)e
        \\&=e
        \\&={g_{a,b}}^{-1}\psi_a(e)g_{a,b}
        \\&=\Sigma((a,b))\psi_{h((a,b))}(\Sigma(b*\gamma))g_{h((a,b)),h(b*\gamma)}.
    \end{align*}

    For all $((a_1,b_1),(a_2,b_2))\in E^{(2)}(\text{Lk}^\gamma) \subset E^{(2)}(\mathcal{Y}(\gamma))$, where $b_1=b_2a_2$,
    \begin{align*}
        &\Sigma_{t((a_1,b_1))}(l_{(a_1,b_1),(a_2,b_2)})\Sigma(a_1a_2, b_2)
        \\=&\psi_{a_1b_1}(e){g_{a_1a_2, b_2}}^{-1}
        \\=&{g_{a_1a_2, b_2}}^{-1}
        \\&\overset{(\pot)}{=}{g_{a_1,b_1}}^{-1}\psi_{a_1}({g_{a_2,b_2}}^{-1})g_{a_1,a_2}
        \\=&\Sigma((a_1,b_1))\psi_{h((a_1,b_1))}(\Sigma((a_2,b_2)))g_{h((a_1,b_1)),h((a_2,b_2))}.
    \end{align*}
\end{proof}

We are now ready to prove Corollary \ref{cor:thm_localCoG}, which we recall below:
\corthmlocalCoG*

\begin{proof}
By Proposition \ref{prop_pi1induced},
the morphism $\Sigma:L(\mathcal{Y}(\gamma))\to G(\mathcal{Y})$ from Construction \ref{construction_morphismfunctorial}
induces a homomorphism on the level of fundamental groups such that
$l\mapsto \Sigma_\gamma(l)$ for all $l\in L_\gamma$. Moreover, by \cite[Theorem III.$\mathcal{C}$.3.7]{bh}, there exists isomorphisms \begin{align*}\bar{\Psi}:\pi_1(G(\mathcal{Y}),\bar{T})&\longrightarrow \pi_1(G(\mathcal{Y}),\gamma),
\\
\Psi:\pi_1(L(\mathcal{Y}(\gamma)),T)&\longrightarrow \pi_1(L(\mathcal{Y}(\gamma)),\gamma)
\end{align*}
that restrict to the identity on the generators $L_\gamma$ and $G_\gamma$ respectively. Hence, the diagram
\begin{center}
    \begin{tikzcd}[row sep=huge,column sep=huge]
    L_\gamma \arrow[r,"(\iota_T)_\gamma"] \arrow[d,swap,"\Sigma_\gamma"] &
\pi_1(L(\mathcal{Y}(\gamma)),T) \arrow[r,"\Psi"]\arrow[d] &
\pi_1(L(\mathcal{Y}(\gamma)),\gamma) \arrow[d,"\Sigma_*"]
\\
G_\gamma \arrow[r,"(\iota_{\bar{T}})_\gamma"] & \pi_1(G(\mathcal{Y}),\bar{T}) \arrow[r,"\bar{\Psi}"] & \pi_1(G(\mathcal{Y}),\gamma)
  \end{tikzcd}
\end{center}
commutes. Observe that injectivity of $\Sigma_*$ implies injectivity of $\Sigma_\gamma\circ(\iota_{\bar{T}})_\gamma$. Thus, ${(\iota_{\bar{T}})}_\gamma$ is injective.

Now assume that $G(\mathcal{Y})$ is developable. The morphism $\Sigma:L(\mathcal{Y}(\gamma))\to G(\mathcal{Y})$ has the property that $\Sigma_\gamma=\text{id}_{L_\gamma}$. Thus, injectivity of $(\iota_{\bar{T}})_\gamma= (\iota_{\bar{T}})_\gamma\circ \Sigma_\gamma$ implies injectivity of $\Sigma_*$.
\end{proof}

\section{Immersions of complexes of groups}\label{section_immersion} 
In this section, we construct a new notion of an \textit{immersion} of complexes of groups and prove that \textit{locally isometric} immersions into non-positively curved complexes of groups mimic the behaviour of locally isometric immersions into non-positively curved metric spaces. This is the content of Theorem \ref{thm_immersion}.

\begin{definition}
\label{def_immersion}
    A morphism of complexes of groups $\phi:H(\mathcal{Y})\to G(\mathcal{X})$ over a morphism of scwols $f:\mathcal{Y}\to \mathcal{X}$ is an \textit{immersion} if it satisfies the following geometric and algebraic conditions:
    \begin{enumerate}
        \item\label{geom} (Geometric) The geometric realization of the map between local developments\footnote{We refer the reader to Proposition \ref{prop_morphismlocaldevelopments} for a description of this morphism.}
        $$|\Phi_\sigma|: \text{St}(\tilde{\sigma})\longrightarrow \text{St}(\widetilde{f(\sigma)})$$ is an embedding.
        \item\label{alg} (Algebraic) The local homomorphism $\phi_\sigma: H_\sigma\to G_{f(\sigma)}$ is injective.
    \end{enumerate}
\end{definition}

In \cite[Definition 3.1]{amar}, Martin defines an immersion of complexes of groups for the case where geometric realizations of underlying scwols are simplicial complexes. In particular, the geometric realization of the underlying morphism between scwols $|f|:|\mathcal{Y}|\to |\mathcal{X}|$ is a simplicial immersion.
We point out that our notion of an immersion introduced here is more general.

\begin{example}
    A covering of complexes of groups, defined in \cite[Definition III.$\mathcal{C}$.5.1]{bh}, is an immersion of complexes of groups.
\end{example}

\begin{example}
    The morphism $\Sigma:L(\mathcal{Y}(\gamma))\to G(\mathcal{Y})$ from Construction \ref{construction_morphismfunctorial} is an immersion of complexes of groups.
\end{example}

\begin{lemma}\label{lemma_coset_immersion}
Definition \ref{def_immersion} (1) is equivalent to the following condition: 

For each $j\in E(\mathcal{X})$ and each $\sigma=t(a)\in  V(\mathcal{Y})$ where $a\in f^{-1}(j)$, the map
        $$\bigsqcup_{\substack{a\in f^{-1}(j),
        \\t(a)=\sigma}} 
        \rightQ{H_\sigma}{\xi_a(H_{i(a)})} 
        \longrightarrow 
        \rightQ{G_{f(\sigma)}}{\psi_{j}(G_{i(j)})}$$ 
        induced by $h \mapsto \phi_{\sigma}(h)\phi(a)$
        is injective.
\end{lemma}
\begin{proof}
We first show that the condition on cosets implies Definition \ref{def_immersion} (1). On $\sigma$, $\Phi_\sigma:\sigma\mapsto f(\sigma)$. This is clearly injective. By the definition of morphisms of scwols, $\Phi_\sigma$ is a bijection on the set of morphisms of the form $b*\sigma\in V(\text{Lk}^\sigma)\times \{\sigma\}$.
Injectivity of $\Phi_\sigma$ on objects $b\in V(\text{Lk}^\sigma)$ come from the non-degeneracy of $f$ and the injectivity of morphisms $b\in E(\mathcal{Y})$, where $i(b)=\sigma$. This implies, again by non-degeneracy, that $\Phi_\sigma$ is injective on $E(\text{Lk}^\sigma)$. Moreover, the first factor of $\Phi_\sigma$ restricted to $\text{Lk}_{\tilde{\sigma}}$ is defined by the coset map.
Hence, $\Phi_\sigma$ is injective. 

Now assume that the coset condition does not hold. As we have seen, by non-degeneracy, $\Phi_\sigma$ is always injective on $\sigma$, $V(\text{Lk}^\sigma)$, $V(\text{Lk}^\sigma)\times \{\sigma\}$ and $ E(\text{Lk}^\sigma)$. 
For some $j\in E(\mathcal{X})$, let $a_1, a_2\in f^{-1}(j)$ such that $t(a_1)=t(a_2)=\sigma\in V(\mathcal{Y})$, and let $h_i\xi_{a_i}(H_{i(a_i)})\in \rightQ{H_\sigma}{\xi_{a_i}(H_{i(a_i)})}$ for $i=1,2$ be distinct cosets such that
\begin{gather*}\label{eqn_failedcosetcondition}
    \phi_\sigma(h_1)\phi(a_1)\psi_j(G_{i(j)})=\phi_\sigma(h_2)\phi(a_2)\psi_j(G_{i(j)}).
\end{gather*}
This implies that there exists distinct $(h_1\xi_{a_1}(H_{i(a_1)}), a_1), (h_2\xi_{a_2}(H_{i(a_2)}), a_2)\in V(\text{Lk}_{\tilde{\sigma}})$ such that
\begin{align*}
    \Phi_\sigma((h_1\xi_{a_1}(H_{i(a_1)}), a_1))&=(\phi_\sigma(h_1)\phi(a_1)\psi_j(G_{i(j)}), j)\\&=(\phi_\sigma(h_2)\phi(a_2)\psi_j(G_{i(j)}), j)\\&=\Phi_\sigma((h_2\xi_{a_2}(H_{i(a_2)}), a_2)),
\end{align*}
so $\Phi_\sigma$ is not injective on $V(\text{Lk}_{\tilde{\sigma}})$.
We have thus proven that when the coset condition holds, $\Phi_\sigma$ must be injective on $V(\text{Lk}_{\tilde{\sigma}})$.
\end{proof}

We conclude this section with our last definition:
\begin{definition}
    A morphism $\phi:H(\mathcal{Y}) \to G(\mathcal{X})$ between two metrized complexes of groups is \textit{locally isometric} if for all $\sigma\in V(\mathcal{Y})$, the map on geometric realizations of local developments $$|\Phi_\sigma|: \text{St}(\tilde{\sigma})\longrightarrow \text{St}(\widetilde{f(\sigma)})$$
    has isometric image.
\end{definition}

We are now ready to prove Theorem \ref{thm_immersion}, which we recall below:
\thmimmersion*

\begin{proof}
Since $G(\mathcal{X})$ is non-positively curved, $|\mathcal{X}|$ is an $M_\kappa-$polyhedral complex for some $\kappa<0$. Thus, $|\mathcal{Y}|$ is metrized into an $M_\kappa-$polyhedral complex by $|f|$ (see Remark \ref{remark_metrized_morphism}). By our notion of an immersion, the map$$|\Phi_\sigma|: \text{st}(\tilde{\sigma})\to \text{st}(\widetilde{f(\sigma)})$$ is an isometric embedding for each $\sigma\in V(\mathcal{Y})$. 
Thus, each $\text{st}(\tilde{\sigma})$ inherits non-positive curvature from $\text{st}(\widetilde{f(\sigma)})$.
By Definition \ref{def_npcCoG}, $H(\mathcal{Y})$ is a non-positively curved complex of groups. By Theorem \ref{thm_NPC}, $H(\mathcal{Y})$ is developable.

Moreover, by Proposition \ref{prop_embeddingofstars}, there is an embedding of stars into developments $|D(\mathcal{Y},\iota_T)|$ and $|D(\mathcal{X},\iota_{T'})|$, given by $$\text{st}(\tilde{\sigma})\longmapsto\text{st}(\bar{\sigma})$$ and $$\text{st}(\widetilde{f(\sigma)})\longmapsto\text{st}(\bar{f(\sigma)})$$ respectively. Since $\phi$ is a locally isometric immersion of complexes of groups, $$|\Phi_\sigma|: \text{St}(\tilde{\sigma})\longrightarrow \text{St}(\widetilde{f(\sigma)})$$ is an isometric embedding. Hence, the map $|\Phi|: |D(\mathcal{Y},\iota_T)|\to |D(\mathcal{X},\iota_{T'})|$ is a locally isometric immersion between metric spaces. By \cite[Proposition II.4.14]{bh}, this lifts to an isometric embedding. Proposition \ref{thm_D(Y,iota_T)} tells us that $|D(\mathcal{Y},\iota_T)|$ and $ |D(\mathcal{X},\iota_{T'})|$ are simply connected. Thus, $|\Phi|$ is an isometric embedding.
    
It remains to show that the induced homomorphism $\phi_*$ on fundamental groups is injective.
If there exists an element of the kernel of $\phi_*$ that does not fix a point, then $|\Phi|$ fails to be an isometry. Hence, such an element must stabilize an object, i.e. it is contained in a local group of $H(\mathcal{Y})$. Since local homomorphisms of immersions are injective, such an element must be trivial.
\end{proof}
\subsection{Acknowledgements}
The author would like to thank her academic siblings Carl Tang, Darius Alizadeh and Ping Wan for peer editing, as well as Jean Pierre Mutanguha for giving many helpful comments on an earlier version. Last but not least, she would like to thank her advisor Daniel Groves for his endless patience and guidance through the research process.
\medskip
\bibliographystyle{amsalpha}
\bibliography{bib.bib}

\end{document}

%% file: macro.tex
\usepackage{amsmath}
\usepackage{tikz}
\usetikzlibrary{matrix}

\usepackage{hyperref}
\hypersetup{colorlinks=true, linkcolor=RubineRed, citecolor=LimeGreen,
urlcolor=SkyBlue}
\usepackage[hypcap=false]{caption}

\usepackage{amssymb,accents}
\usepackage{amsthm}
\usepackage{enumitem}
\usepackage{tikz-cd}
\usepackage[mathscr]{euscript}
\usepackage{mathtools}
\usepackage{caption}
\usepackage{subfigure}
\usepackage{thm-restate}

\usepackage{tikzsymbols}
\usepackage{textcomp}

\tikzset{
math to/.tip={Glyph[glyph math command=rightarrow]},
loop/.tip={Glyph[glyph math command=looparrowleft, swap]},
loop'/.tip={Glyph[glyph math command=looparrowleft]},
 weird/.tip={Glyph[glyph math command=Rrightarrow, glyph length=1.5ex]},
  pi/.tip={Glyph[glyph math command=pi, glyph length=1.5ex, glyph axis=0pt]},
}

\newcommand\loops\ell

\newtheorem*{claim}{Claim}

\newtheorem{theorem}{Theorem}[section]
\newtheorem{lemma}[theorem]{Lemma}

\newtheorem{proposition}[theorem]{Proposition}

\newtheorem*{claim*}{Claim}

\newtheorem{assumption}{Assumption}

\theoremstyle{definition}
\newtheorem{remark}[theorem]{Remark}
\newtheorem{definition}[theorem]{Definition}
\newtheorem{example}[theorem]{Example}

\newtheorem{construction}[theorem]{Construction}

\renewcommand\bar\overline

\newcommand{\rightQ}[2]{\left.\raisebox{.2em}{$#1$}\middle/\raisebox{-.2em}{$#2$}\right.}
\newcommand{\leftQ}[2]{\left.\raisebox{-.2em}{$#2$}\middle\backslash\raisebox{.2em}{$#1$}\right.}

\usepackage{lipsum}                     
\usepackage{xargs}                      

\usepackage{todonotes}
\newcommandx{\unsure}[2][1=]{\todo[linecolor=red,backgroundcolor=red!25,bordercolor=red,#1]{#2}}
\newcommandx{\change}[2][1=]{\todo[linecolor=blue,backgroundcolor=blue!25,bordercolor=blue,#1]{#2}}
\newcommandx{\info}[2][1=]{\todo[linecolor=OliveGreen,backgroundcolor=OliveGreen!25,bordercolor=OliveGreen,#1]{#2}}
\newcommandx{\improvement}[2][1=]{\todo[linecolor=Plum,backgroundcolor=Plum!25,bordercolor=Plum,#1]{#2}}
\newcommandx{\thiswillnotshow}[2][1=]{\todo[disable,#1]{#2}}

\usepackage[dvipsnames]{xcolor}